\documentclass{article}

\usepackage{amsmath}
\usepackage{amssymb}
\usepackage{amsthm}
\usepackage{graphicx}
\usepackage{enumerate}
\usepackage{xcolor}
\usepackage{subfigure}
\usepackage{marginnote}

\usepackage[normalem]{ulem}

\usepackage{cancel}

\newcommand{\CondsC}{C.1--C.5}

\newtheorem{thm}{Theorem}

\newtheorem{lem}{Lemma}
\newtheorem{remark}{Remark}

\numberwithin{equation}{section}

\newcommand{\RR}{\mathbb{R}}
\newcommand{\T}{\mathbb{T}}
\newcommand{\sgn}{\text{sign}}
\newcommand{\pmm}{{+\!/\!-}}
\newcommand{\X}{\mathcal{X}}
\newcommand{\bt}{\bar{t}_0}
\newcommand{\by}{\bar{y}_0}
\newcommand{\tP}{\tilde{P}}

\newcommand{\tr}{\tilde{r}}
\newcommand{\te}{\tilde{\varepsilon}}
\newcommand{\tS}{\tilde{\Sigma}}

\newcommand{\wt}{\widehat{t}_0}
\newcommand{\tc}{\tilde{c}}

\newcommand{\epsmin}{\varepsilon_{\text{min}}}

\usepackage{rotating}

\graphicspath{
{figures/one_block/}
{figures/phase_portrait_unperturbed_non_linear/}
{figures/poincare_map/}
{figures/splitted_separatrices_non_linear/}
{figures/phase_portrait_unperturbed_dissipative/}
{figures/impact_sequence/}
{figures/poincare_map_dissipative/}
{figures/simulations/linear/subharmonic_orbits/5-1_periodic_orbit/eps1.6565e-2/}
{figures/simulations/linear/dissipative_orbits/5-1_po/5-1_ratio0.07_cond1/}
{figures/simulations/linear/dissipative_orbits/5-1_po/5-1_ratio0.07_cond2/}
{figures/eps-r_5-1_curves/}
}

\begin{document}

\title{The Melnikov method and subharmonic orbits in a piecewise smooth system}
\author{A. Granados, S.J. Hogan and T.M. Seara}
\date{}
\maketitle

\begin{abstract}
In this work we consider a two-dimensional piecewise smooth system, defined in
two domains separated by the switching manifold $x=0$. We assume that there exists a
piecewise-defined continuous Hamiltonian that is a first integral of the system.
We also suppose that the system possesses an invisible fold-fold at the
origin and two heteroclinic orbits connecting two hyperbolic critical points on either 
side of $x=0$. Finally, we assume that the region closed by
these heteroclinic connections is fully covered by periodic orbits surrounding
the origin, whose periods monotonically increase as they approach the
heteroclinic connection.\\
When considering a non-autonomous ($T$-periodic) Hamiltonian perturbation of
amplitude $\varepsilon$, using an impact map, we rigorously prove that,
for every $n$ and $m$ relatively prime and $\varepsilon>0$ small enough, there
exists a $nT$-periodic orbit impacting $2m$ times with the switching manifold at
every period if a modified subharmonic Melnikov function possesses a simple
zero. We also prove that, if the orbits are 
discontinuous when they cross $x=0$, then all
these orbits exist if the relative size of $\varepsilon>0$ with respect to the
magnitude of this jump is large enough.\\
We also obtain similar conditions for the splitting of the heteroclinic
connections.
\end{abstract}


\section{Introduction}
The Melnikov method provides tools to determine the persistence of periodic
orbits and homoclinic/heteroclinic connections for planar regular systems under
non-autonomous periodic perturbations~\cite{GucHol83}. This persistence is
guaranteed by the existence of simple zeros of a certain function, the
subharmonic Melnikov function and the Melnikov function, respectively.  In this
work we extend these classical results to a class of piecewise smooth differential
equations generalizing a mechanical impact model. In such systems, the perturbation 
typically models an external forcing and, hence, affects a second order
differential equation. However, in this work, we allow for a general
periodic Hamiltonian perturbation, potentially influencing both velocity and
acceleration. Note that no symmetry is assumed in either the perturbed or
unperturbed system.\\

The unperturbed system is defined in two domains separated by a
switching manifold where the impacts occur, and possesses one hyperbolic critical
point on either side of it. We distinguish between two situations regarding
the unperturbed system. In the first one, which we call conservative, two
heteroclinic trajectories connect both hyperbolic points, and surround a region
completely covered by periodic orbits including the origin. In the second one,
we introduce an energy dissipation at the impacts, which is modeled by an
algebraic condition that forces the solutions to undergo a discontinuity every
time they cross the switching manifold. Then, the origin becomes a global attractor
and none of these objects can exist for the unperturbed system.\\

For a smooth system, the classical Melnikov method considers fixed (or periodic)
points of the time $T$ stroboscopic map, where $T$ is the period of the
perturbation.  However,
for our class of system, such a map becomes unwieldy because one has to
check the number of times that the flow crosses the switching manifold, which
is a priori unknown and can even be arbitrarily large. Instead, using the switching
manifold as a Poincar\'e section and adding time as variable, we consider the first return Poincar\'e map, the 
so-called \textit{impact map}. This map is smooth and hence we can use classical
perturbation theory to rigorously prove sufficient conditions for the existence of 
periodic orbits. In the conservative case, these conditions
turn out to be same ones given by the classical Melnikov method, so
extending it to a class of piecewise smooth systems
(Theorem~\ref{theo:subharmonic_melnikov_conservative}). In addition, we
rigorously prove that the simple zeros of the
subharmonic Melnikov function can guarantee the existence of periodic orbits
when the trajectories are forced to be discontinuous due to the loss of energy
at impact (Theorem~\ref{theo:dissipative_subharmonic_orbits}).\\
In addition, the impact map could also be used to prove the existence of
invariant KAM tori in the system. After writing the system in action-angle
variables, these ideas were applied in~\cite{KunKupYou97} to a different system
to prove the existence of such tori.

To prove the existence of heteroclinic connections
for the perturbed case, it is sufficient to look for the intersection with the switching manifold of the
stable and unstable manifolds~\cite{BruKoc91},~\cite{Hog92}. In this way, we rigorously extend the
classical Melnikov method for heteroclinic connections to a class of piecewise smooth
systems. When the loss
of energy is considered, we prove that the zeros of the Melnikov function
guarantee the existence of transversal heteroclinic intersections.
Both results are given in Theorem~\ref{theo:intersection_of_separatrices}.\\

This paper is organized as follows. In \S\ref{sec:system_decription} we describe
the class of system that we consider, state some notation and introduce tools 
needed for this work. In \S\ref{sec:existence_of_subharmonic_orbits}, we prove the
existence of periodic orbits distinguishing between the conservative and
dissipative cases. \S\ref{sec:intersection_of_separatrices} is devoted to heteroclinic
connections. Finally, in \S\ref{sec:rocking_block}, we use the example of the
rocking block to illustrate the results obtained regarding the periodic orbits,
and compare with the work of~\cite{Hog89}.

\section{System description}\label{sec:system_decription}
\subsection{General system definition}\label{sec:system_definition}
We divide the plane into two sets,
\begin{align*}
&S^+=\left\{ (x,y)\in\RR^2\,\vert\,x>0 \right\}\\
&S^-=\left\{ (x,y)\in\RR^2\,\vert\,x<0 \right\}
\end{align*}
separated by the switching manifold 
\begin{equation}
\Sigma=\Sigma^+\cup\Sigma^-\cup (0,0)
\label{eq:boundary}
\end{equation}
where
\begin{align*}
\Sigma^+&=\Big\{ (x,y)\in\RR^2\,\vert \, x>0 \Big\}\\
\Sigma^-&=\Big\{ (x,y)\in\RR^2\,\vert \, x<0 \Big\}.
\end{align*}
We consider the piecewise smooth system
\begin{equation}
\left( 
\begin{array}{c}
\dot{x}\\\dot{y}
\end{array}
 \right)=
\left\{ 
\begin{aligned}
&\X_0^+(x,y)+\varepsilon \X_1^+(x,y,t)&&\text{if }(x,y)\in S^+\\
&\X_0^-(x,y)+\varepsilon \X_1^-(x,y,t)&&\text{if }(x,y)\in S^-
\end{aligned}
 \right.
\label{eq:general_field}
\end{equation}
We assume $\X_0^\pm\in C^\infty(\RR^2)$ and
$\X_1^\pm(x,y,t)\in C^\infty(\RR^3)$, although this can be relaxed to less
regularity in $S^\pm$ and $S^\pm\times\RR$, respectively.\\
System (\ref{eq:general_field}) is a Hamiltonian system associated
with a $C^0$ piecewise smooth Hamiltonian of the form
\begin{equation}
H_\varepsilon(x,y,t)=H_0(x,y)+\varepsilon H_1(x,y,t).
\label{eq:general_Hamiltonian}
\end{equation}
The unperturbed $C^0(\RR^2)$ Hamiltonian $H_0$ is a classical Hamiltonian given by
\begin{equation}
H_0(x,y):=\frac{y^2}{2}+V(x):=\left\{ 
\begin{aligned}
&H_0^+(x,y):=\frac{y^2}{2}+V^+(x)&&\text{if }(x,y)\in S^+\cup \Sigma\\
&H_0^-(x,y):=\frac{y^2}{2}+V^-(x)&&\text{if }(x,y)\in S^-
\end{aligned}
\right.
\label{eq:unperturbed_Hamiltonian}
\end{equation}
with $V^\pm\in C^\infty (\RR^2)$ satisfying $V^+(0)=V^-(0)$.\\
Similarly, the non-autonomous $T$-periodic $C^0(\RR^3)$ perturbation,
$\varepsilon H_1$, is given by
\begin{equation*}
H_1(x,y,t):=\left\{
\begin{aligned}
&H_1^+(x,y,t)&&\text{if }(x,y)\in S^+\cup \Sigma^+\\
&H_1^-(x,y,t)&&\text{if }(x,y)\in S^-
\end{aligned}\right.
\end{equation*}
fulfilling $H_1^+(0,y,t)=H_1^-(0,y,t)$ $\forall (y,t)\in\RR^2$.\\
Then, the relation between (\ref{eq:general_field}) and
(\ref{eq:general_Hamiltonian}) is given by
\begin{equation}
\begin{aligned}
\X_0^++\varepsilon \X_1^+&=J\nabla (H_0^++\varepsilon H_1^+)\\
\X_0^-+\varepsilon\X_1^-&=J\nabla (H_0^-+\varepsilon H_1^-),
\end{aligned}
\label{eq:hamiltonian_properties}
\end{equation}
where $J$ is the usual symplectic matrix
\begin{equation*}
J=\left( \begin{array}{cc}
0&1\\-1&0
\end{array} \right).
\end{equation*}
We assume that the phase portrait of the unperturbed system
(\ref{eq:general_field}) ($\varepsilon=0$) is topologically equivalent to the
one shown in Fig.~\ref{fig:phase_portrait_unperturbed}, which we make precise in
the following hypotheses.
\begin{enumerate}[C.1]
\item There exist two hyperbolic critical points $z^+ \equiv (x^+,y^+)\in S^+$
and $z^- \equiv (x^-,y^-)\in S^-$ of saddle type belonging to the energy level
\begin{equation}
\Big\{ (x,y)\,\vert\,H_0(x,y)=c_1>0 \Big\}.
\label{eq:critical_points_energy_level}
\end{equation}
\item The origin is an invisible fold-fold of centre type~\cite{GuaSeaTei11}, such that $H_0(0,0)=0$.
\item There exist two heteroclinic orbits given by $W^u(z^-)=W^s(z^+)$
and\linebreak $W^u(z^+)=W^s(z^-)$ surrounding the origin and contained in the
energy level~(\ref{eq:critical_points_energy_level}).
\item The region between both heteroclinic orbits is fully covered by periodic
orbits surrounding the origin given by
\begin{equation}
\Lambda_c=\Big\{ (x,y)\in\RR^2\,\vert\,H_0(x,y)=c \Big\}
\label{eq:periodic_orbits}
\end{equation}
with $0<c<c_1$, and $\Lambda_c$ intersects $\Sigma$ transversally exactly twice.
\item The period of $\Lambda_c$ is a regular function of $c$ with strictly
positive derivative for $0<c<c_1$.
\end{enumerate}
Note that, as the unperturbed Hamiltonian $H_0$ is $C^\infty$ in $S^+$ and
$S^-$, the fact that the heteroclinic orbits are in the energy level
$H_0(x,y)=c_1$ follows automatically from hypothesis C.1. However, we include it
explicitly for clarity.\\

We wish to determine which of these objects and characteristics persist and
which are destroyed when the small non-autonomous $T$-periodic perturbation
$\varepsilon H_1$ is considered. Of interest is the splitting of the
separatrices and the persistence of periodic orbits.  In the smooth case, these
answers are given completely by the classical Melnikov method~\cite{GucHol83}.
Hence, it is natural to check whether these classical tools are still valid for
the piecewise smooth system presented above and if any changes to the method are
necessary.\\
Another interesting question that can be addressed with a similar approach is
the existence of $2$-dimensional invariant tori of system
(\ref{eq:general_field}) (see \cite{KunKupYou97,Kun00}).

\begin{figure}
\begin{center}
\includegraphics[width=0.8\textwidth]{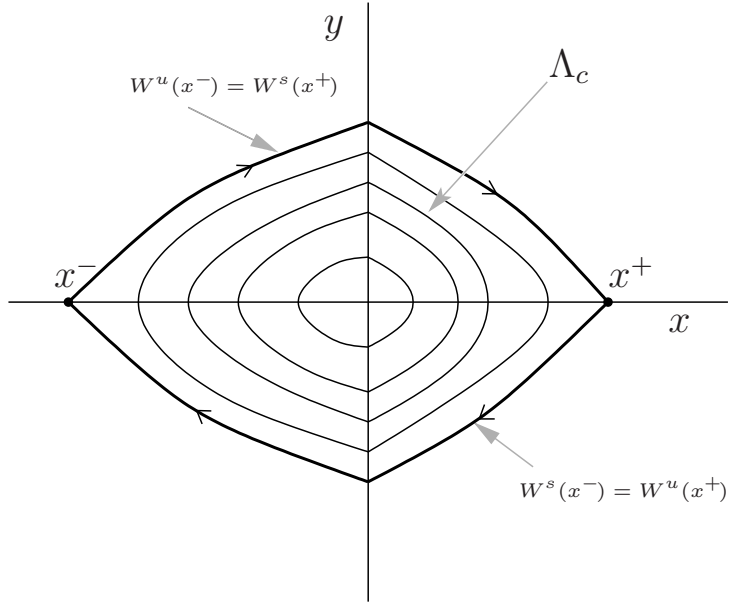}
\end{center}
\caption{Phase portrait for the unperturbed system (\ref{eq:general_field}).}
\label{fig:phase_portrait_unperturbed}
\end{figure}

\subsection{Poincar\'e impact map}\label{sec:impact_map}
To study system~(\ref{eq:general_field}) we will proceed as in \cite{Hog89}
using the Poincar\'e impact map. We consider the extended phase space
$\RR^2\times\RR$ adding time as a system variable and equation $\dot{t}=1$ to
Eq.~(\ref{eq:general_field}).  As the perturbation is periodic, this
time variable is usually defined in $\mathbb{T}=\RR/T$; however, it will be more useful
for us to consider $\RR$ instead.  We want to study the motion in the region
surrounded by the heteroclinic orbit,  so we consider in this extended phase-space
the Poincar\'e section
\begin{equation}
\tS^+=\left\{
(0,y,t)\in\RR^2\times\RR\,\vert\,0<y<\sqrt{2c_1} \right\}.
\label{eq:poincare_section}
\end{equation}
To simplify the notation, as the first coordinate in $\tS^+$ is always $0$, we
will omit its repetition whenever this does not lead to confusion.  The domain
of the Poincar\'e map is not $\tS^+$ but a suitable open set $U$, that depends on
$\varepsilon$ and, for $\varepsilon=0$, does not contain the heteroclinic
connection.\\

We now define the Poincar\'e impact map
\begin{equation*}
P_\varepsilon:U\subset\tS^+\longrightarrow \tS^+,
\end{equation*}
as follows (see Fig.~\ref{fig:poincare_map}).  
First, using the section
\begin{equation}
\tS^-=\left\{ (0,y,t)\in\RR^2\times\RR\,\vert\, -\sqrt{2c_1}<y<0 \right\},
\label{eq:poincare_section_negative}
\end{equation}
with $(0,y_0,t_0)\in U^+\subset\tS^+$, we define the map
\begin{equation*}
P^+_\varepsilon:U^+\subset\tS^+\longrightarrow\tS^-,
\end{equation*}
as
\begin{equation}
P^+_{\varepsilon}(y_0,t_0)=\left( \Pi_y\left( \phi^+\left(t_1; t_0,
0,y_0,\varepsilon\right)
\right),t_1 \right)
\label{eq:Pp}
\end{equation}
where $\phi^+(t;t_0,x,y,\varepsilon)$ is the flow associated with system
(\ref{eq:general_field}) restricted to $S^+$, and $t_1>t_0$ is the smallest value
of $t$ satisfying the condition
\begin{equation}
\Pi_x\left(\phi^+\big(t_1;t_0,0,y_0,\varepsilon\big)\right)=0.
\label{eq:t1}
\end{equation}
Similarly, we consider
\begin{equation*}
P^-_\varepsilon:U^-\subset\tS^-\longrightarrow\tS^+
\end{equation*}
for $(0,y_1,t_1)\in U^-\subset\tS^-$ defined by
\begin{equation}
P^-_\varepsilon(y_1,t_1)=\left( \Pi_y\left( \phi^-\left(t_2; t_1,
0,y_1,\varepsilon \right) \right),t_2 \right)
\label{eq:Pm}
\end{equation}
where $\phi^-(t;t_1,x,y,\varepsilon)$ is the flow associated with
(\ref{eq:general_field}) restricted to $S^-$, and $t_2>t_1$ is the smallest
value of $t$ satisfying the condition
\begin{equation}
\Pi_x\left(\phi^-\big(t_2;t_1,0,y_1,\varepsilon\big)\right)=0.
\label{eq:t2}
\end{equation}
Then the Poincar\'e impact map is defined as the composition
\begin{equation}
\begin{array}{cccc}
P_\varepsilon:&U\subset\tS^+&\longrightarrow&\tS^+\\
&{\text{\footnotesize($y_0,t_0)$}}&\longmapsto&\text{\footnotesize$P_\varepsilon^-\!\!\circ\!
P_\varepsilon^+\!(y_0,t_0)$}\\
\end{array}
\label{eq:poincare_map}
\end{equation}
\begin{figure}
\begin{center}
\includegraphics[width=0.6\textwidth]{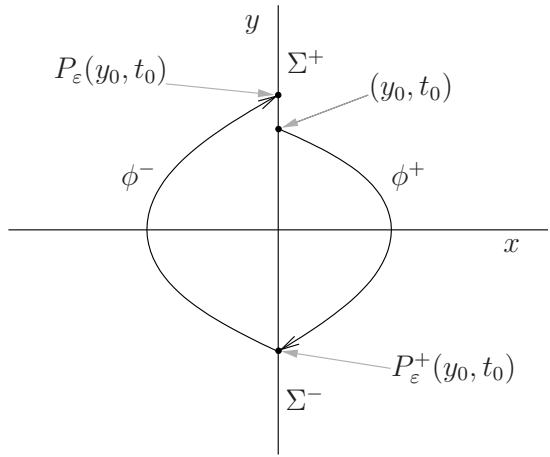}
\end{center}
\caption{Poincar\'e impact map (\ref{eq:poincare_map}) represented schematically.}
\label{fig:poincare_map}
\end{figure}%
Notice that, as assumed in C.4, for the unperturbed flow all initial
conditions in $\Sigma^+$ lead to periodic orbits surrounding the origin.
Hence, we can give a closed expression for $P_0$, the Poincar\'e impact map when
$\varepsilon=0$. Let
\begin{equation}
\alpha^\pm(\pm y)=\pm2\mspace{-30mu}\mathop{\int}_0^{(V^\pm)^{-1}(h)}
\mspace{-30mu}\frac{1}{\sqrt{2(h-V^\pm(x))}}dx,\quad h=H_0(0,\pm y)=\frac{y^2}{2}
\label{eq:half_period_unperturbed_orbit}
\end{equation}
be the time needed by an orbit of the unperturbed system with initial condition
$(0,\pm y)\in\Sigma^\pm$ to reach $\Sigma^\mp$. In the unperturbed case, the
orbit with initial condition $(0,y)\in\Sigma^+$ has period
\begin{equation}
\alpha(y)=\alpha^+(y)+\alpha^-(-y).
\label{eq:period_unperturbed_orbit}
\end{equation}
Then the Poincar\'e impact map when $\varepsilon=0$ is defined in the
whole $\tS^+$, and can be written as
\begin{equation}
P_0(y_0,t_0)=(y_0,t_0+\alpha(y_0)).
\label{eq:unperturbed_impact_map1}
\end{equation}
Thus, if $\varepsilon$ is small enough, the perturbed trajectories starting at
$\tS^+$ cross $\tS^+$ again. The Poincar\'e impact map is well defined, and
is as smooth as the flow restricted to $S^+$ and $S^-$.\\
Note that in the symmetric case, $V^+(x)=V^-(-x)$,
$\alpha^+(y)=\alpha^-(-y)$ is half the period of the unperturbed periodic
orbit with initial condition $(0,y)\in\Sigma^+$.

\subsection{Coefficient of restitution}\label{sec:rest_coeff}
As the name of the previous map suggests, it is typically used to deal with
systems with impacts, as is the case of the mechanical example of section
\ref{sec:rocking_block}. In order to include the loss of energy at the impact,
one considers a coefficient of restitution, $r\in(0,1]$, that reduces the velocity,
$y$, at every impact. More precisely, if a trajectory crosses $\Sigma$
transversally at some point $(0,y_B)$ at $t=t_B$, then the state is replaced by
$(0,ry_B)$ at a later time $t_A$ to proceed with the evolution of the system. In
other words, the system slides along $\Sigma$ from $(0,y_B)$ to
$(0,ry_B)$ during time $t_A-t_B$ and
\begin{equation}
y(t_A)=ry(t_B).
\label{eq:restitution_coefficient}
\end{equation}
For the rest of this article we will assume that the loss of energy is produced
instantaneously and hence $t_A=t_B$. Thus, there is no sliding along $\Sigma$
and the trajectory jumps from $(0,y_B)$ to $(0,ry_B)$.\\

Clearly, when such a condition is introduced to a system of the type
(\ref{eq:general_field}), the unperturbed system ($\varepsilon=0$) is no longer
conservative, the origin becomes a global attractor and none of the conditions
\CondsC~hold. In particular, the orbits with initial conditions on the unstable
manifolds $W^u(z^-)$ and $W^u(z^+)$ tend to the origin and can not intersect the
stable manifolds $W^s(z^+)$ and $W^s(z^-)$, respectively (see
Fig.~\ref{fig:phase_portrait_unperturbed_dissipative}).
\begin{figure}
\begin{center}
\includegraphics[width=0.8\textwidth]{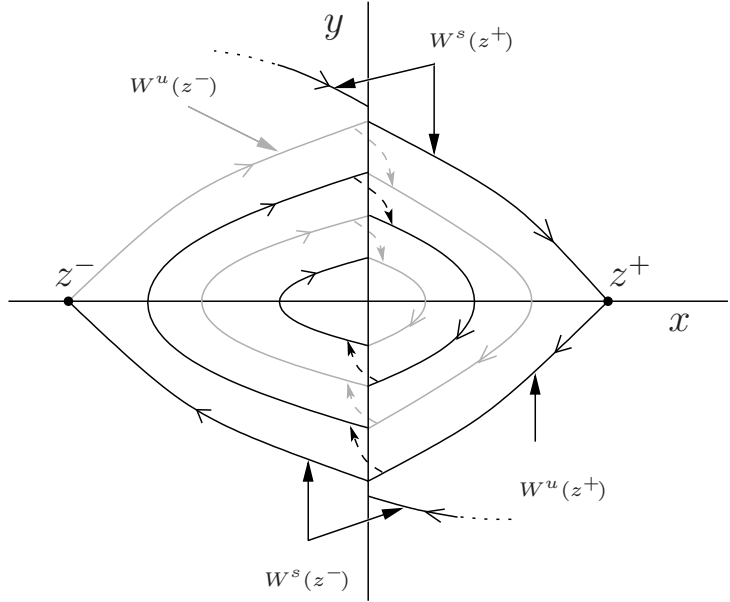}
\end{center}
\caption{Stable and unstable manifolds of system (\ref{eq:general_field}) for
$r<1$ and $\varepsilon=0$.}
\label{fig:phase_portrait_unperturbed_dissipative}
\end{figure}
\begin{figure}
\begin{center}
\includegraphics[width=0.6\textwidth]{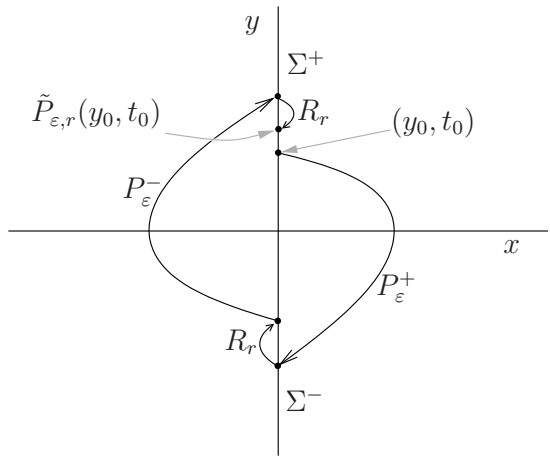}
\end{center}
\caption{Impact map for $r<1$ and $\varepsilon>0$.}
\label{fig:poincare_map_dissipative}
\end{figure}

Although periodic orbits surrounding the origin are not possible for the
unperturbed case if $r<1$, they may exist if $\varepsilon>0$. However,
roughly speaking, as these orbits will have to overcome the loss of energy, the
magnitude of the forcing will not be allowed to be arbitrarily small. We will
make a precise statement of this fact in \S\ref{sec:dissipative_subharmonic}
(see also \cite{Hog89}).\\

To study the existence of periodic orbits we will use again the impact map,
which can also be defined for $r<1$ as (see
Fig.~\ref{fig:poincare_map_dissipative})
\begin{equation}
\tP_{\varepsilon,r}(y_0,t_0):=R_r\circ P_\varepsilon^-\circ R_r\circ P_\varepsilon^+(y_0,t_0)
\label{eq:dissipative_impact_map}
\end{equation}
where
\begin{equation*}
R_r(y_0,t_0)=(ry_0,t_0).
\end{equation*}
Note that $\tP_{\varepsilon,r}$ is as smooth as the flow restricted to $S^\pm$,
since it is the composition of smooth maps.\\

Using Eqs.~(\ref{eq:half_period_unperturbed_orbit}) and
(\ref{eq:period_unperturbed_orbit}), the impact map, $\tP_{\varepsilon,r}$, for
$\varepsilon=0$ and $r<1$ can be written as
\begin{equation}
\tP_{0,r}(y_0,t_0)=\left(r^2y_0,t_0+\alpha^+\left(y_0\right)+\alpha^-(-ry_0)\right).
\label{eq:unperturbed_impact_map2}
\end{equation}
Note that, for any $\varepsilon>0$,
\begin{equation*}
\tP_{\varepsilon,1}(y_0,t_0)=P_\varepsilon(y_0,t_0).
\end{equation*}

\subsection{Some formal definitions and notation}\label{sec:definitions}
Up to now, we have considered separately the solutions of system
(\ref{eq:general_field}) in $S^+$ and $S^-$until they reach the
switching manifold $\Sigma$. Given an initial condition $(x_0,y_0,t_0)$, one can
extend the definition of a solution, $\phi(t;t_0,x_0,y_0,\varepsilon,r)$, of
system (\ref{eq:general_field}),(\ref{eq:restitution_coefficient}) for all $t\ge
t_0$ by properly concatenating $\phi^+$ or $\phi^-$ whenever the flow crosses
$\Sigma$ transversally.  Depending on the sign of $x_0$, one applies either
$\phi^+(t;t_0,x_0,y_0,\varepsilon)$ or $\phi^-(t;t_0,x_0,y_0,\varepsilon)$ until
the trajectory reaches $\Sigma$, and then one
applies~(\ref{eq:restitution_coefficient}). If $x_0=0$, one proceeds similarly
depending on the sign of $y_0$. This is because $\dot{x}=y+O(\varepsilon)$ is
always an equation of the flow and the orbits twist clockwise.

In this work, we will mainly use solutions with initial conditions
$(0,y_0,t_0)\in\tS^+$. In that case,  we define the sequence of impacts
$(0,y_{\varepsilon,r}^i,t_{\varepsilon,r}^i)$ (see
Fig.~\ref{fig:sequence_of_impacts}), if they exist, as
\begin{equation}
(y^i_{\varepsilon,r},t^i_{\varepsilon,r})=
\left\{ 
\begin{array}{ll}
R_r\circ
P_\varepsilon^-(y^{i-1}_{\varepsilon,r},t^{i-1}_{\varepsilon,r})&\text{if
}y^{i-1}_{\varepsilon,r}<0\\
R_r\circ
P_\varepsilon^+(y^{i-1}_{\varepsilon,r},t^{i-1}_{\varepsilon,r})&\text{if
}y^{i-1}_{\varepsilon,r}>0
\end{array}
\right.,
\label{eq:impact_sequence}
\end{equation}
with $(y^0_{\varepsilon,r},t^0_{\varepsilon,r})=(y_0,t_0)$ and
$P_\varepsilon^\pm$ defined in~(\ref{eq:Pp}) and~(\ref{eq:Pm}). Notice that the
sequence (\ref{eq:impact_sequence}) will be finite if the flow reaches
$\Sigma$ a finite number of times only.\\
For the unperturbed case, for any point $\left( 0,y_0,t_0 \right)\in\tS^+$, the
sequence (\ref{eq:impact_sequence}) becomes
\begin{equation}
(y^i_{0,r},t^i_{0,r}):=\left\{
\begin{array}{ll}
\left( r^{i}y_0,t_{0,r}^{i-1}+\alpha^-\left(
-r^{i-1}y_0\right)\right)&\text{if }i\ge2\text{ even}\\
\left(-r^{i}y_0,t_{0,r}^{i-1}+\alpha^+\left(
r^{i-1}y_0\right)\right)&\text{if }i\ge1\text{ odd}
\end{array}
\right.,
\label{eq:unperturbed_impact_sequence}
\end{equation}
where $\alpha^\pm$ are defined in
Eq.~(\ref{eq:half_period_unperturbed_orbit}).\\
\begin{figure}
\begin{center}
\includegraphics[width=0.6\textwidth]{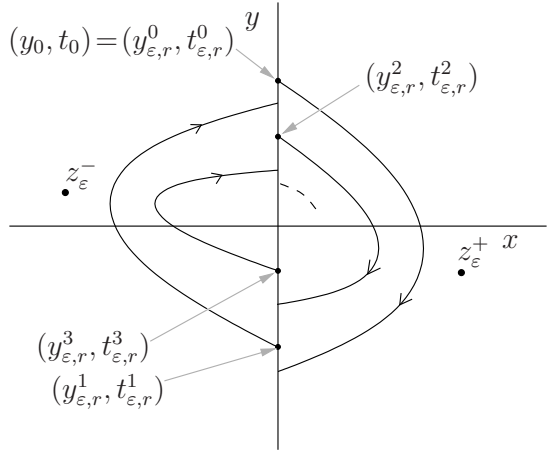}
\end{center}
\caption{Sequence of impacts for $r<1$ and $\varepsilon>0$.}
\label{fig:sequence_of_impacts}
\end{figure}%
Once the impacts $(y_{\varepsilon,r}^i,t_{\varepsilon,r}^i)$ are defined, the
solution of the non-autonomous system
\linebreak(\ref{eq:general_field}),(\ref{eq:restitution_coefficient}) with
initial condition $(0,y_0,t_0)\in\tS^+$ is given as
\begin{equation}
\phi(t;t_0,0,y_0,\varepsilon,r):=\left\{
\begin{array}{ll}
\phi^+(t;t_{\varepsilon,r}^{2i},0,y_{\varepsilon,r}^{2i},\varepsilon)&\text{if
}t_{\varepsilon,r}^{2i}\le t <t^{2i+1}_{\varepsilon,r}\\
\\
\phi^-(t;t_{\varepsilon,r}^{2i+1},0,y_{\varepsilon,r}^{2i+1},\varepsilon)&\text{if
}t_{\varepsilon,r}^{2i+1}\le t
<t_{\varepsilon,r}^{2i+2}
\end{array}
\right.,\;i\ge0.
\label{eq:general_solution}
\end{equation}
Note that in the case when the number of impacts is finite, we take the
last interval of time to be infinitely long.\\

In the rest of the paper we will generally distinguish between the
conservative ($r=1$) and dissipative ($r<1$) cases. We will omit the parameter
$r$ in the flow $\phi$ whenever we refer to $r=1$.\\
Note that we have only defined the solution of the system for an initial
condition $(0,y_0,t_0)\in\tS^+$. Given $(0,y_0,t_0)\in\tS^-$, one defines
similarly this solution by just properly shifting the subscripts of
$t_\varepsilon^i$ in (\ref{eq:general_solution}). In addition, it is
possible to extend precisely this definition to an arbitrarily initial condition
$(x_0,y_0,t_0)$.\\

As is usual when dealing with Hamiltonian systems, we will use the
unperturbed Hamiltonian to measure the distance between states. In addition, as
we are dealing with a perturbation problem, we will frequently use
expansions in powers of $\varepsilon$. In this case, the integral of the Poisson
brackets of the Hamiltonians $H_1$ and $H_0$ typically provides a compact
expression for the linear terms in $\varepsilon$. Given $m\ge1$,
$(0,y_0,t_0)\in\tS^+$ and its impact sequence
$(0,y^i_{\varepsilon,r},t_{\varepsilon,r}^i)$, $0\le i\le 2m$, for the non-smooth
system (\ref{eq:general_field}),(\ref{eq:restitution_coefficient}) when $r\le
1$, we introduce
\begin{equation}
\begin{aligned}
\int_{t_0}^{t_{\varepsilon,r}^{2m}}\left\{ H_0,H_1 \right\}&\left( \phi\left(
t;t_0,0,y_0,\varepsilon,r \right),t \right)dt\\
&:=\sum_{i=0}^{m-1}\Bigg(\int_{t_{\varepsilon,r}^{2i}}^{t_{\varepsilon,r}^{2i+1}}\left\{
H_0^+,H_1^+ \right\}\left(
\phi^+(t;t_{\varepsilon,r}^{2i},0,y_{\varepsilon,r}^{2i},\varepsilon),t
\right)dt\\
&+\int_{t_{\varepsilon,r}^{2i+1}}^{t^{2i+2}_{\varepsilon,r}}\left\{ H_0^-,H_1^- \right\}\left( \phi^-(
t;t_{\varepsilon,r}^{2i+1},0,y_{\varepsilon,r}^{2i+1},\varepsilon),t
\right)dt\Bigg)
\end{aligned}
\label{eq:poisson_brackets_integral}
\end{equation}
where $\left\{ Q\left( x,y \right),R\left( x,y \right) \right\}=\frac{\partial
Q}{\partial x}\frac{\partial R}{\partial y}-\frac{\partial Q}{\partial
y}\frac{\partial R}{\partial x}$ is the usual Poisson bracket of the
Hamiltonians $Q$ and $R$.\\

The next Lemma provides an expression for $H_0\left( \phi\left(
t^{2m}_{\varepsilon,r};t_0,0,y_0,\varepsilon,r \right) \right)$ which we will
use below. 
\begin{lem}
Let $m\ge1$ and $(0,y_0,t_0)\in\tS^+$, and let
$(0,y^i_{\varepsilon,r},t^i_{\varepsilon,r})$, $i=0,\dots,2m$, be its associated
impact sequence as defined in (\ref{eq:impact_sequence}). Then,
\begin{equation}
\begin{aligned}
H_0(0,y_{\varepsilon,r}^{2m})&-H_0\left( 0,y_0
\right)=
r^2\Bigg[\varepsilon \int_{t_0}^{t_{\varepsilon,r}^{2m}}\left\{ H_0,H_1 \right\}\left( \phi\left(
t;t_0,0,y_0,\varepsilon,r \right),t
\right)dt\\
&+\sum_{i=0}^{2m-1}\left( H_0\left(0,y^{i}_{\varepsilon,r}\right)-H_0\left(
0,\frac{y^{i}_{\varepsilon,r}}{r}
\right) \right)\Bigg].
\end{aligned}
\label{eq:barrows_formula1}
\end{equation}
\end{lem}
\begin{proof}
The proof of this Lemma comes from a straightforward application of the
fundamental theorem of calculus to the smooth functions $H_0^\pm\left(
\phi^\pm\left( t;t_0,0,\pm y_0,\varepsilon \right) \right)$, using the fact that
\begin{align*}
H_0(0,y_{\varepsilon,r}^{2m})&=H_0\left( r\phi\left(
t_{\varepsilon,r}^{2m};t_0,0,y_0,\varepsilon,r
\right) \right)\\
&=r^2H_0\left(\phi\left(
t_{\varepsilon,r}^{2m};t_0,0,y_0,\varepsilon,r
\right) \right),
\end{align*}
taking into account the intermediate gaps induced by the impact
condition~(\ref{eq:restitution_coefficient}) and using the fact that
\begin{equation*}
\frac{d}{dt}H_0^\pm\left( \phi^\pm\left( t;t^*,x^*,y^*,\varepsilon \right)
\right)=\varepsilon\left\{ H_0^\pm,H_1^\pm \right\}\left( \phi^\pm\left(
t;t^*,x^*,y^*,\varepsilon
\right) \right)
\end{equation*}
for any $(x^*,y^*)\in S^\pm\cup\Sigma^\pm$ and $t\ge t^*$ such that $\phi^\pm\left(
t;t^*,x^*,y^*,\varepsilon \right)\in S^\pm$.
\end{proof}
The following Lemma gives us an expression for the expansion in powers of
$\varepsilon$  of $H_0(0,y_{\varepsilon,r}^{2m})-H_0\left( 0,y_0 \right)$
which we will use in \S\ref{sec:existence_of_subharmonic_orbits}.\\
\begin{lem}\label{lem:barrow}
Let $m\ge1$ and $(0,y_0,t_0)\in\tS^+$, and let
$(0,y^i_{\varepsilon,r},t^i_{\varepsilon,r})$, $i=0,\dots,2m$, be its associated
impact sequence as defined in (\ref{eq:impact_sequence}). Then, if
$\varepsilon\simeq0$, the Taylor expansion of expression
(\ref{eq:barrows_formula1}) becomes
\begin{equation}
\begin{aligned}
H_0(0,y_{\varepsilon,r}^{2m})-H_0\left( 0,y_0
\right)&=\frac{y_0^2}{2}(r^{4m}-1)+\varepsilon G^{m}(y_0,t_0)\\
&+O(\varepsilon^2)+O(\varepsilon\left( r-1
\right))
\end{aligned}
\label{eq:barrows_formula_expanded}
\end{equation}
where
\begin{equation}
G^{m}(y_0,t_0)=
\int_{0}^{m\alpha(y_0)}\left\{ H_0,H_1 \right\}\left( \phi\left(
t;0,0,y_0,0\right),t+t_0
\right)dt
\label{eq:G}
\end{equation}
and $\alpha(y_0)$ is given in~(\ref{eq:period_unperturbed_orbit}).
\end{lem}
\begin{proof}
The independent term of the expansion is found by noting that, if
$\varepsilon=0$, from expression~(\ref{eq:unperturbed_impact_sequence}), one has
$H_0(0,y^{i}_{0,r})=H_0(0,\frac{y^{i+1}_{0,r}}{r})$.  Hence all the terms in the
sum of Eq.~(\ref{eq:barrows_formula1}) cancel each other except for the first
and the last one. This, in combination with the fact that
$H_0(0,y)=\frac{y^2}{2}$, gives the first term in
Eq.~(\ref{eq:barrows_formula_expanded}). For the linear term in $\varepsilon$,
one first obtains 
\begin{align*}
&r^2\Bigg[
\int_{t_0}^{t_{\varepsilon,r}^{2m}}\left\{ H_0,H_1 \right\}\left( \phi\left(
t;t_0,0,y_0,0,r \right),t
\right)dt\\
&+(r^2-1)\sum_{i=1}^{2m-1}\left(
\frac{d}{d\varepsilon}\left(H_0\left( \left(
y_{\varepsilon,r}^i,t_{\varepsilon,r}^i \right) \right)\right)_{\vert
\varepsilon=0}\right)\Bigg].
\end{align*}
Then, by applying $m$ times Eq.~(\ref{eq:unperturbed_impact_map2}), one has that
$t_{\varepsilon,1}^{2m}=t_0+m\alpha(y_0)+O(\varepsilon)$. Thus, by expanding
this for $r$ near $1$ and $\varepsilon$ near $0$ and noting that the unperturbed
flow is autonomous and hence $\phi(t;t_0,0,y_0,0)=\phi(t-t_0;0,0,y_0,0)$, one
gets expression~(\ref{eq:G}).
\end{proof}
\begin{remark}\label{remark:r1}
If in Eq.~(\ref{eq:G}) we take $\alpha(y_0)=\frac{nT}{m}$, then we recover the classical
Melnikov function for the subharmonic orbits \cite{GucHol83} with the modified
integral~(\ref{eq:poisson_brackets_integral}).
\end{remark}

\section{Existence of subharmonic orbits}\label{sec:existence_of_subharmonic_orbits}
\subsection{Conservative case, $r=1$: Melnikov method for subharmonic orbits}\label{sec:conservative_subharmonic}
Let us consider system (\ref{eq:general_field}) neglecting the loss of energy at
impact ($r=1$ in Eq.~(\ref{eq:restitution_coefficient})). According to \CondsC,
for $\varepsilon=0$, this system possesses a continuum of periodic orbits,
$\Lambda_c$ in Eq.~(\ref{eq:periodic_orbits}), surrounding the origin. Our main
goal in this section is to investigate the persistence of these orbits when the
(periodic) non-autonomous perturbation is considered ($\varepsilon>0$). The
classical Melnikov method for subharmonic orbits, which here, in principle, does
not apply, provides sufficient conditions for the persistence of periodic orbits
for a smooth system with an equivalent, smooth, unperturbed phase portrait.

The period of the orbits $\Lambda_c$ tends to infinity as they approach the
heteroclinic orbit. More precisely, if $q_c(t)$ is the periodic orbit satisfying
$q_c(0)=(0,y_0)$ with $H_0(0,y_0)=c$, its period $\alpha(y_0)$ tends to infinity
as $c\to c_1$ (see formula~(\ref{eq:period_unperturbed_orbit})). As we are
interested in finding periodic orbits for $0<\varepsilon\ll1$, we will use the
unperturbed periodic solutions as $\varepsilon$-close approximations to them.
In general, such perturbation results are valid only for finite time and
therefore, from now on, we will restrict ourselves to a set of the form
\begin{equation}
\tilde{\Sigma}^+_{\tc}=\left\{ (0,y,t)\in\tS^+\,\vert\, 0\le y \le
\tc \right\},
\label{eq:compact_set}
\end{equation}
for a fixed $\tc$ satisfying $0\le \tilde{c}<\sqrt{2c_1}$. Note that, if
$(0,y_0,t_0)\in \tS_{\tc}^+$ then $\alpha(y_0)$ is uniformly bounded
($\alpha(y_0)<\alpha(\tc)$). However, following \cite{GucHol83}, it is also
possible to extend the method for all the periodic orbits up to the heteroclinic
connection.\\

To look for periodic orbits we will use the impact map defined
in~(\ref{eq:poincare_map}).  In terms of this map, a point in $U\subset\tS^+$ will lead
to a periodic orbit of period $nT$ if it is a solution of the equation
\begin{equation}
P^m_\varepsilon(y_0,t_0)=(y_0,t_0+nT),
\label{eq:periodic_orbit_condition}
\end{equation}
for some $m$. We take $m$ to be the smallest integer such that
(\ref{eq:periodic_orbit_condition}) is satisfied. In that case, $\phi\left(
t;t_0,0,y_0,\varepsilon \right)$ will be a periodic orbit of period $nT$, which
crosses the switching manifold $\Sigma$ exactly $2m$ times. We will call this an
$(n,m)$-periodic orbit. Then for $(n,m)$-periodic orbits
with $\varepsilon>0$ we have the following result analogous to the smooth case 
\begin{thm}\label{theo:subharmonic_melnikov_conservative}
Consider a system as defined in (\ref{eq:general_field}) satisfying \CondsC, and
let $\alpha(y_0)$ be the function defined
in~(\ref{eq:half_period_unperturbed_orbit})-(\ref{eq:period_unperturbed_orbit}).
Assume that the point $(0,\by,\bt)\in\tS_{\tc}^+$ satisfies
\begin{enumerate}[H.1]
\item $\alpha(\by)=\frac{n}{m}T$, with $n,m\in\mathbb{Z}$ relatively prime
\item $\bt\in[0,T]$ is a simple zero of 
\begin{equation}
M^{n,m}(t_0)=\int_0^{nT}\left\{ H_0,H_1
\right\}
(q_{c}(t),t+t_0)
dt,\;c=H_0(0,\by),
\label{eq:subharmonic_Melnikov_function}
\end{equation}
\end{enumerate}
where $q_c(t)=\phi\left( t;0,0,\by \right)$ is the periodic orbit such that
$\alpha(\by)=\frac{nT}{m}$.\\
Then, there exists $\varepsilon_0$ such that, for every
$0 < \varepsilon < \varepsilon_0$, one can find $y_0^*$ and $t_0^*$ such that
$\phi(t;t_0^*,0,y_0^*,\varepsilon)$ is an $(n,m)$-periodic orbit.
\end{thm}

\begin{proof}
The proof of the result comes from a straightforward application of the implicit
function theorem to equation~(\ref{eq:periodic_orbit_condition}). Let us fix
$n$ and $m$ relatively prime. We replace equation
(\ref{eq:periodic_orbit_condition}) by
\begin{equation}
\left( 
\begin{array}{c}
H_0\left( 0,\Pi_{y_0}(P_\varepsilon^m(y_0,t_0)) \right)\\
\Pi_{t_0}\left( P^m_\varepsilon(y_0,t_0) \right)
\end{array}
 \right)-
\left( 
\begin{array}{c}
H_0(0,y_0)\\ t_0+nT
\end{array}
\right)=\left( \begin{array}{c}
0\\0
\end{array} \right).
\label{eq:periodic_orbit_condition2}
\end{equation}
That is, we use the Hamitlonian $H_0$ to measure the distance between the points
$\left(0,\Pi_{y_0}\left(P^m_\varepsilon\left(y_0,t_0\right)\right)\right)$ and
$\left(0,y_0\right)$.\\
Using the second equation in~(\ref{eq:periodic_orbit_condition2}) we have
\begin{align*}
\Pi_{y_0}\left(P_\varepsilon^m\left(y_0,t_0\right)\right)&=\Pi_{y}\left(\phi(t_0+nT;t_0,0,y_0,\varepsilon)\right)\\
0&=\Pi_x\left( \phi\left(t_0+ nT;t_0,0,y_0,\varepsilon \right)
\right).
\end{align*}
This allows us to rewrite Eq.~(\ref{eq:periodic_orbit_condition2}) as
\begin{equation}
\left(
\begin{array}{c}
H_0(\phi(t_0+nT;t_0,0,y_0,\varepsilon))-H_0(0,y_0)\\
\Pi_{t_0}\left(P_{\varepsilon}^m\left(y_0,t_0\right)\right)-nT-t_0
\end{array}
\right)=
\left( \begin{array}{c}
0\\0
\end{array} \right).
\label{eq:periodic_orbit_condition3}
\end{equation}
We expand Eq.~(\ref{eq:periodic_orbit_condition3}) in powers of $\varepsilon$.
Using Eq.~(\ref{eq:unperturbed_impact_map1}), the second component
of~(\ref{eq:periodic_orbit_condition3}) becomes
\begin{equation}
\Pi_{t_0}\left( P^m_\varepsilon\left( y_0,t_0 \right)
\right)-t_0-nT=m\alpha(y_0)-nT+O(\varepsilon)=0,
\label{eq:second_component_periodic_orbit_condition}
\end{equation}
where $\alpha(y_0)$ is the period of the periodic orbit $q_{c}(t)$,
$c=H_0(0,y_0)$, given in Eq.~(\ref{eq:period_unperturbed_orbit}).\\
On the other hand, using Lemma~\ref{lem:barrow} and
noting that
\begin{equation*}
\Pi_{y_0}\left(P_\varepsilon^m \left( y_0,t_0
\right)\right)=y_{\varepsilon,1}^{2m},
\end{equation*}
the first equation in~(\ref{eq:periodic_orbit_condition3}) can be written as
\begin{align*}
H_0(0,\Pi_{y_0}&(P_\varepsilon^m(y_0,t_0)))-H_0(0,y_0)\\
&=\varepsilon\int_{0}^{m\alpha(y_0)}\left\{
H_0,H_1\right\}(\phi(t;0,0,y_0,0),t+t_0)dt+O(\varepsilon^2)\\
&=\varepsilon G^{m}(y_0,t_0)+O(\varepsilon^2).
\end{align*}
where $G^m(y_0,t_0)$ is given in~(\ref{eq:G}). Hence,
Eq.~(\ref{eq:periodic_orbit_condition3}) finally becomes
\begin{equation}
F_{n,m}(y_0,t_0,\varepsilon):=
\left(
\begin{array}{c}
G^{m}(y_0,t_0)+O(\varepsilon)\\
m\alpha(y_0)-nT+O(\varepsilon)\\
\end{array}
\right)=
\left( \begin{array}{c}
0\\0
\end{array} \right),
\label{eq:periodic_orbit_condition4}
\end{equation}
where the order in $\varepsilon$ of the first component has been reduced and,
thus, the implicit function theorem can be applied to
Eq.~(\ref{eq:periodic_orbit_condition4}). Therefore, one needs
\begin{enumerate}
\item $F_{n,m}(\by,\bt,0)=(0,0)^T$
\item $\det(D_{y_0,t_0}F(\by,\bt,0))\ne0$, where $D_{y_0,t_0} \equiv D$ 
is the Jacobian with respect to the variables $y_0$ and $t_0$.
\end{enumerate}
The first condition is satisfied by noting in
Eq.~(\ref{eq:periodic_orbit_condition4}) that $\by$ has to be such that
$\alpha(\by)=\frac{nT}{m}$ and $\bt$ a zero of the subharmonic Melnikov function
\begin{equation*}
M^{n,m}(t_0):=G^{m}(\by,t_0)=\int_{0}^{nT}\left\{
H_0,H_1\right\}(q_c(t),t+t_0)dt,
\end{equation*}
where $q_c(t)$, $c=H_0(0,\by)$, is the unperturbed periodic orbit of period
$\frac{nT}{m}$ such that $q_c(0)=(0,y_0)$, and therefore $q_c(t)=\phi\left(
t;0,0,\by,0 \right)$.\\
In addition, for $\varepsilon=0$, $DF_{n,m}$ is given by 
\begin{equation*}
DF_{n,m}(y_0,t_0,0)=
\left( 
\begin{array}{cc}
\frac{\partial G^{m}}{\partial y_0} & \frac{\partial G^{m}}{\partial t_0}\\
m\alpha'(y_0) & 0
\end{array}
 \right).
\end{equation*}
By C.5, $\alpha'(y_0)\ne0$,
and the second condition is satisfied if $\bt$ is a simple zero of the subharmonic
Melnikov function, $M^{n,m}(t_0)$, which completes hypothesis \emph{H.2}.\\
Finally, applying the implicit function theorem
to~(\ref{eq:periodic_orbit_condition4}) at $(y_0,t_0,\varepsilon)=(\by,\bt,0)$,
there exists $\varepsilon_0>0$ such that, if $0 < \varepsilon < \varepsilon_0$, then
there exist unique $y_0^*(\varepsilon)$ and $t_0^*(\varepsilon)$ solutions of
the equation (\ref{eq:periodic_orbit_condition2}), which have the form
\begin{align*}
y_0^*=\by+O(\varepsilon)\\
t_0^*=\bt+O(\varepsilon).
\end{align*}
Hence, the orbit $\phi\left( t;t^*_0,0,y^*_0,\varepsilon \right)$ is an
$(n,m)$-periodic orbit, as it has period $nT$ and impacts $2m$ times with the
switching manifold $\Sigma$ in every period.
\end{proof}
\begin{remark}\label{remark:epsilon_0}
The upper bound $\varepsilon_0$ given in the theorem depends on $n$ and $m$.
However, for every fixed $m$, it is possible to obtain
$\varepsilon_0(m)$, such that for $\varepsilon < \varepsilon_0(m)$, we can 
apply the theorem for all $n$ such that
$\alpha^{-1}(\frac{nT}{m})\in\tS_{\tc}^{+}$. This is because the
approximation of the perturbed flow by the unperturbed periodic orbit is
performed $m$ times beyond the period of the unperturbed periodic orbit.
\end{remark}
\begin{remark}\label{rem:algorithm}
The proof of the result provides us with a constructive method
to find the initial condition for $nT$-periodic orbits for $\varepsilon>0$.
\begin{enumerate}
\item Given $n$ and $m$, find $\by$ such that $\alpha(\by)=\frac{n}{m}T$ using
Eq.~(\ref{eq:period_unperturbed_orbit}).
\item Find $\bt$ such that $M^{n,m}(t_0)$ has a simple zero at $t_0=\bt$.
\item Use $(\by,\bt)$ as seed to solve Eq.~(\ref{eq:periodic_orbit_condition2})
numerically.
\end{enumerate}
\end{remark}
\begin{lem}
The subharmonic Melnikov function~(\ref{eq:subharmonic_Melnikov_function}) is
either identically zero or generically possesses at least one simple zero.
\end{lem}
\begin{proof}
The proof comes from the fact that $M^{n,m}(t_0)$ has average
\begin{equation*}
<M^{n,m}(t_0)>=\frac{1}{T}\int_0^TM^{n,m}(t_0)dt_0
\end{equation*}
equal to zero.
\begin{align*}
<M^{n,m}(t_0)>&=\frac{1}{T}\int_0^T\int_0^{nT}\left\{ H_0,H_1 \right\}\left( q_c(t),t+t_0
\right)dtdt_0\\
&=\frac{1}{T}\int_0^{nT}\int_0^T\left\{ H_0,H_1 \right\}\left( q_c(t),t+t_0 \right)dt_0dt\\
&=\int_0^{nT}\left\{ H_0,<H_1> \right\}(q_c(t))dt.
\end{align*}
Recalling that $\alpha(y_0)=\frac{nT}{m}$ (see
(\ref{eq:half_period_unperturbed_orbit})-(\ref{eq:period_unperturbed_orbit})) and letting
\begin{align*}
&q_c^+(t)=\phi^+(t;0,0,y_0,0)\\
&q_c^-(t)=\phi^-(t;\alpha^+(y_0),0,-y_0,0),
\end{align*}
$<M^{n,m}(t_0)>$ can be written as
\begin{align*}
&m\left(\int_0^{\alpha^+(y_0)}\left\{ H_0^+,<H_1^+> \right\}\left( q_c^+(t)
\right)dt+\int_{\alpha^+(y_0)}^{\frac{nT}{m}}\left\{ H_0^-,<H_1^->
\right\}\left( q_c^-(t) \right)dt\right)\\
&=-m\left(\int_0^{\alpha^+(y_0)}\frac{d}{dt}\left( <H_1^+>\left( q_c^+(t) \right)
\right)dt +\int_{\alpha^+(y_0)}^{\frac{nT}{m}}\frac{d}{dt}\left( <H_1^->\left(
q_c^-(t) \right)\right)dt\right)\\
&=-m\Bigg( <H_1^+>\left( q_c^+(\alpha^+(y_0))\right)-<H_1^+>\left(q_c^+(0)\right)\\
&+<H_1^->\left(q_c^-({\textstyle\frac{nT}{m}})\right)-<H_1^->\left(q_c^-\left(
\alpha^+(y_0)
\right)\right) \Bigg)=0.
\end{align*}
\end{proof}

Note that, if $M^{n,m}(t_0) \equiv 0$ then a second order analysis is required to study
the existence of periodic orbits.

\subsection{Dissipative case, $r<1$}\label{sec:dissipative_subharmonic}
We now focus on the situation when the coefficient of restitution $r$ introduced
in \S\ref{sec:rest_coeff} is considered. As already mentioned, for
$\varepsilon=0$ the origin is a global attractor and hence none of the periodic
orbits studied in the previous section exists if the amplitude of the
perturbation is small enough.  However, as was shown in \cite{Hog89} for the
rocking block model, for $\varepsilon$ large enough an infinite number periodic
orbits surrounding the origin can exist. This was studied analytically and
numerically for the rocking block model under symmetry assumptions for the
particular case $m=1$. Here, our goal is to relate the periodic orbits existing
for the dissipative case to those which exist for $r=1$ in the general system
(\ref{eq:general_field}),(\ref{eq:restitution_coefficient}). As will be shown
below, all the periodic orbits given by
Theorem~\ref{theo:subharmonic_melnikov_conservative} can also exist for the
dissipative case, when $r<1$ is small enough compared with $\varepsilon>0$. In
other words, we generalise in this section the result presented for the
conservative case.\\

As in \S\ref{sec:conservative_subharmonic}, in order to obtain the initial
conditions of a $(n,m)$-periodic orbit for $r<1$, one has to solve the equation
\begin{equation}
\tilde{P}^m_{\varepsilon,r}(y_0,t_0)=(y_0,t_0+nT),
\label{eq:dissipative_nm_periodic_orbit_condition}
\end{equation}
where $\tilde{P}_{\varepsilon,r}$, is defined in
Eq.~(\ref{eq:dissipative_impact_map}). The next result states that, under
certain conditions regarding $r$ and $\varepsilon$,
Eq.~(\ref{eq:dissipative_nm_periodic_orbit_condition}) can be solved.

\begin{thm}\label{theo:dissipative_subharmonic_orbits}
Consider system (\ref{eq:general_field}),(\ref{eq:restitution_coefficient}). Let
$(0,\by,\bt)\in\tS^+$ be such that $\alpha(\by)=\frac{nT}{m}$, with $n$ and $m$
relatively prime, and $\bt$ a simple zero of the subharmonic
Melnikov function~(\ref{eq:subharmonic_Melnikov_function}). There exists $\rho$
such that, given $\te,\tr>0$ satisfying $0<\frac{\tr}{\te}<\rho$, there exists
$\delta_0$ such that, if $\varepsilon=\te\delta$ and $r=1-\tr\delta$, then
$\forall\,0<\delta<\delta_0$ there exists $(y_0^*,t_0^*)$ which is a solution of
Eq.~(\ref{eq:dissipative_nm_periodic_orbit_condition}).  Moreover,
$y_0^*=\by+O(\delta)$, $t_0^*=\bt+O(\delta)+O(\tr/\te)$ and the
solution $(y_0^*,t_0^*)$ tends to the one given in
Theorem~\ref{theo:subharmonic_melnikov_conservative} as $r\to1^-$.
\end{thm}
\begin{proof}
As in the conservative case, we use the unperturbed Hamiltonian to measure the
distance between points in $\Sigma$. Then,
Eq.~(\ref{eq:dissipative_nm_periodic_orbit_condition}) can be rewritten as
\begin{equation}
\left( 
\begin{array}{c}
H_0\left( 0,\Pi_{y_0}(\tilde{P}_\varepsilon^m(y_0,t_0)) \right)\\
\Pi_{t_0}\left( \tilde{P}^m_\varepsilon(y_0,t_0) \right)
\end{array}
 \right)-
\left( 
\begin{array}{c}
H_0(0,y_0)\\ t_0+nT
\end{array}
\right)=\left( \begin{array}{c}
0\\0
\end{array} \right).
\label{eq:dissipative_nm_periodic_orbit_condition2}
\end{equation}
As in Theorem~\ref{theo:subharmonic_melnikov_conservative}, we proceed by expanding
this equation in powers of $\varepsilon$ and $r-1$
using~(\ref{eq:barrows_formula_expanded}) and
(\ref{eq:G}) obtaining
\begin{equation}
\left( 
\begin{array}{c}
\frac{y_0^2}{2}(r^{4m}-1)+\varepsilon G^m(y_0,t_0)
+O(\varepsilon^2)+O(\varepsilon(r-1))\\
\displaystyle
\sum_{i=0}^{m-1}\alpha^+(r^{2i}y_0)+\sum_{i=0}^{m-1}\alpha^-\left(
-r^{2i+1}y_0 \right)
+O(\varepsilon)-nT
\end{array} 
\right)=
\left(\begin{array}{c}
0\\0
\end{array}\right).
\label{eq:dissipative_nm_periodic_orbit_condition3}
\end{equation}
Note that, for $r=1$, Eq.~(\ref{eq:dissipative_nm_periodic_orbit_condition3})
becomes Eq.~(\ref{eq:periodic_orbit_condition3}).

We are interested in studying
Eq.~(\ref{eq:dissipative_nm_periodic_orbit_condition3}) when $1-r$ and
$\varepsilon$ are both small. Therefore, for $\te>0$ and $\tr>0$ we set
\begin{equation}
\varepsilon=\te\delta,\;r=1-\tr\delta,
\label{eq:change_of_parameters}
\end{equation}
where $\delta>0$ is a small parameter. Then
Eq.~(\ref{eq:dissipative_nm_periodic_orbit_condition3}) becomes
\begin{align}
\tilde{F}&_{n,m}(y_0,t_0,\delta):=\nonumber\\
&\left( 
\begin{array}{c}
-2m\tr y_0^2+\te G^m(y_0,t_0)+O(\delta)\\
\displaystyle m\alpha(y_0)+O(\delta)-nT
\end{array} 
\right)=
\left(\begin{array}{c}
0\\0
\end{array}\right).\label{eq:dissipative_nm_periodic_orbit_condition4}
\end{align}
We now need to apply the implicit function theorem to
(\ref{eq:dissipative_nm_periodic_orbit_condition4}).\\
The first step is to solve
Eq.~(\ref{eq:dissipative_nm_periodic_orbit_condition4}) for $\delta=0$. The
second equation gives $\alpha(\by)=\frac{nT}{m}$, as in
Theorem~\ref{theo:subharmonic_melnikov_conservative}.
To solve the first equation, we define
\begin{equation}
f^{n,m}(t_0)=-2m\tr\by^2+\te M^{n,m}(t_0),
\label{eq:wt_equation}
\end{equation}
and $\wt$ will be given by a zero of $f^{n,m}(t_0)$. Assume now that $\bt$ is a
simple zero of $M^{n,m}(t_0)$. As $M^{n,m}(t_0)$ is a smooth periodic function,
it possesses at least one local maximum. Let $t_M$ be the closest value to $\bt$
where $M^{n,m}(t_0)$ possesses a local maximum, and assume
$\left(M^{n,m}\right)'(t_0)\ne0$ for all $t_0$ between $\bt$ and $t_M$. If
$\left(M^{n,m}\right)'(t_0)$ vanishes between $\bt$ and $t_M$, we then take
$t_M$ to be the closest value to $\bt$ such that $\left(M^{n,m}\right)'(t_0)=0$
to ensure that $\left(M^{n,m}\right)'(t_0)\ne0$ between $\bt$ and $t_M$. We then
define $\rho:=\frac{M^{n,m}(t_M)}{2m\by^2}$.  Then, if
\begin{equation}
0<\frac{\tr}{\te}<\rho,
\label{eq:dissipative_nm_periodic_orbit_existence_condition1}
\end{equation}
there exists $\wt$ $\frac{\tr}{\te}$-close to $\bt$ where $f^{n,m}(t_0)$
has a simple zero. Since $\alpha'(\by)>0$, a similar
calculation to the one in Theorem~\ref{theo:subharmonic_melnikov_conservative}
shows that
\begin{equation*}
\det\left(D\tilde{F}_{y_0,t_0}\left(\by,\widehat{t}_0,0\right)\right)\ne0,
\end{equation*}
and hence we can apply the implicit function theorem near
$(y_0,t_0,\delta)=(\by,\widehat{t}_0,0)$ to show that there exists
$\delta_0$ such that, if $0<\delta<\delta_0$, then there exists
\begin{equation*}
(y_0^*,t_0^*)=(\by,\widehat{t}_0)+O(\delta)=\left( \by,\bt
\right)+O(\delta)+O(\tr/\te) 
\end{equation*}
which is a solution of Eq.~(\ref{eq:dissipative_nm_periodic_orbit_condition}).\\
This solution tends to the one given by
Theorem~\ref{theo:subharmonic_melnikov_conservative} when $\tr\to0^+$ . This is
a natural consequence of that fact that
Eq.~(\ref{eq:dissipative_nm_periodic_orbit_condition}) tends to the
Equation~(\ref{eq:periodic_orbit_condition}) as $r\to 1^-$.
\end{proof}

\begin{remark}
In order to determine $\rho$, we have imposed $t_M$ to be the local maximum of
the Melnikov function closest to its simple zero, $\bt$. Instead, one could also
use the absolute maximum so increasing the range given in
Eq.~(\ref{eq:dissipative_nm_periodic_orbit_existence_condition1}). However, in
this case, the values where $\left(M^{n,m}_1\right)'(t_0)=0$ have to be avoided
to ensure that the desired zero of $f^{n,m}(t_0)$ is simple.
\end{remark}

\begin{remark}
Arguing as in Remark~\ref{remark:epsilon_0}, for every $m$ fixed, the constant
$\delta_0(m)$ can be taken such that if $\delta < \delta_0(m)$, there exist 
periodic orbits for all $n$ such that $\alpha(\frac{nT}{m})^{-1}\in\tS$.
\end{remark}

\section{Intersection of the separatrices}\label{sec:intersection_of_separatrices}
We now focus our attention on the invariant manifolds of the saddle
points of system~(\ref{eq:general_field}),(\ref{eq:restitution_coefficient})
when $\varepsilon>0$.  As explained in \S\ref{sec:system_decription}, for
$\varepsilon=0$, there exist two heteroclinic orbits connecting the critical
points $z^\pm$ if $r=1$ (see Fig.~\ref{fig:phase_portrait_unperturbed})
whereas, if $r<1$, the unstable manifolds $W^u(z^\pm)$ spiral discontinuously from $z^\pm$ to the
origin and $W^s(z^\pm)$ becomes unbounded (see
Fig.~\ref{fig:phase_portrait_unperturbed_dissipative}). As we will show, in
both cases, heteroclinic orbits may exist for the perturbed system.\\

For a smooth system with Hamiltonian $K_0(x,y)+\varepsilon K_1(x,y,t)$, the
persistence of homoclinic or heteroclinic connections is achieved by the well
known Melnikov method which states that if the Melnikov function
\begin{equation*}
M(t_0)=\int_{-\infty}^{+\infty}\left\{K_0,K_1\right\}\left(
\phi\left(t;t_0,z_0,0\right),t+t_0
\right)dt,
\end{equation*}
with $z_0=(x_0,y_0)\in W^u(z^-)=W^s(z^+)$, has a simple zero, then the stable
and unstable manifolds intersect for $\varepsilon>0$ small enough
(see \cite{GucHol83}).\\
In this section we will modify the classical Melnikov method  and we will
rigorously prove that it is still valid  for a piecewise smooth system of the
form (\ref{eq:general_field}), even if $r\le1$.\\
There exist in the literature several works where this tool has been used in
particular non-smooth examples, \cite{Hog92,BruKoc91}.
Theorem~\ref{theo:intersection_of_separatrices}
generalises the result stated in \cite{BruKoc91} where the Melnikov method is
shown to work, although the proof there is not complete.\\
The homoclinic version of a piecewise-defined system with a different topology
was studied in \cite{Kun00}, \cite{Kuc07} and \cite{BatFec08}.  However, the
tools developed there do not apply for a system of the type
(\ref{eq:general_field}).\\

We begin by discussing the persistence of objects for $\varepsilon>0$ and
$r\le1$. It is clear that by separately extending the systems
$\X_0^\pm+\varepsilon\X_1^\pm$ to $\RR^2\times\T$, where $\T=\RR/T$, we get two
smooth systems for which the classical perturbation theory holds. It follows
then that, as $z^\pm$ are hyperbolic fixed points, for $\varepsilon>0$ small
enough there exist two hyperbolic $T$-periodic orbits, $\Lambda^\pm_\varepsilon
\equiv \{z_\varepsilon^\pm(\tau); \tau \in [0,T]\}$, with two-dimensional stable
and unstable manifolds $W^{s,u}(\Lambda^\pm_\varepsilon)$.\\
\begin{figure}
\begin{center}
\includegraphics{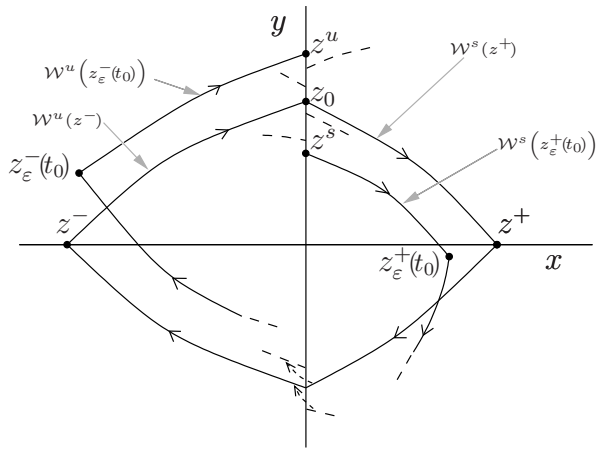}
\end{center}
\caption{Section of the unperturbed and perturbed invariant manifolds for
$t=t_0$.}
\label{fig:perturbed_invariant_manifolds_one_block}
\end{figure}
As the system is non-autonomous, we fix the Poincar\'e section
\begin{equation*}
\Theta_{t_0}=\left\{ (x,y,t_0),\, (x,y)\in\RR^2 \right\},
\end{equation*}
and consider the time $T$ stroboscopic map
\begin{equation*}
\Pi_{t_0}:\Theta_{t_0}\longrightarrow\Theta_{t_0+T},
\end{equation*}
where
\begin{equation*}
\Pi_{t_0}(z)=\phi(t_0+T;t_0,z,\varepsilon,r)
\end{equation*}
and $\phi$ is defined \S\ref{sec:definitions}.\\
This map has $z^\pm_{\varepsilon}(t_0)$ as hyperbolic fixed points with one
dimensional stable and unstable manifolds $W^{s,u}(z^\pm_{\varepsilon}(t_0))$
(see Fig.~\ref{fig:phase_portrait_unperturbed_dissipative}).  Proceeding as in
\cite{BruKoc91}, we fix the section $\Sigma$ defined in (\ref{eq:boundary}) and
study its intersection with the stable and unstable curves
$W^u(z_\varepsilon^-(t_0))$ and $W^s(z_\varepsilon^+(t_0))$. In the unperturbed
conservative case ($\varepsilon=0$ and $r=1$), $W^u(z^-)$ and $W^s(z^+)$
intersect transversally $\Sigma$ in a point $z_0$. The perturbed manifolds,
$W^u(z_\varepsilon^-(t_0)$) and $W^s(z_\varepsilon^+(t_0)$), intersect $\Sigma$
at points $z^u(t_0)$ and $z^s(t_0)$ respectively, $\varepsilon$-close to
$z_0$ (see Fig.~\ref{fig:perturbed_invariant_manifolds_one_block}).  Recalling
the effect of the coefficient of restitution~(\ref{eq:restitution_coefficient})
explained in \S\ref{sec:rest_coeff}, both invariant manifolds will intersect if,
for some $t_0$, one has $rz^u(t_0)=z^s(t_0)$, $r\le1$.  As in \cite{BruKoc91}
and \cite{Hog92}, we use the unperturbed Hamiltonian $H_0(x,y)$ to measure the
distance $\Delta(t_0,\varepsilon,r)$ between $z^u$ and $z^s$
\begin{equation}
\Delta(t_0,\varepsilon,r)=H_0(rz^u(t_0))-H_0(z^s(t_0))=r^2H_0(z^u(t_0))-H_0(z^s(t_0)).
\label{eq:distance_zus}
\end{equation}
We then have the following result.
\begin{thm}\label{theo:intersection_of_separatrices}
Consider system (\ref{eq:general_field}),(\ref{eq:restitution_coefficient}), and
let $z_0=W^s(z^+)\cap\Sigma$. Define the Melnikov function
\begin{equation}
M(t_0)=\int_{-\infty}^{+\infty}\mspace{-10mu}\left\{
H_0,H_1
\right\}\left( \phi\left(t; t_0,z_0,0 \right),t \right)dt,
\label{eq:Melnikovs_function}
\end{equation}
where 
\begin{equation}
\phi(t;t_0,z_0,0)=\left\{
\begin{array}{ll}
\phi^-(t;t_0,z_0,0)&\text{if }t\le t_0\\
\phi^+(t;t_0,z_0,0)&\text{if }t>t_0
\end{array}
\right..
\label{eq:unperturbed_heteroclinic_orbit}
\end{equation}
is the piecewise smooth heteroclinic orbit that exists for $r=1$ and $\varepsilon=0$. Assume that
$M(t_0)$ possesses a simple zero at $\bt$. Then the
following holds.
\begin{enumerate}[a)]
\item If $r=1$, there exists $\varepsilon_0>0$ such that, for every
$0<\varepsilon<\varepsilon_0$, one can find a simple zero $t_0^*=\bt+O(\varepsilon)$ of the function
$\Delta(t_0,\varepsilon,1)$. Hence, the curves
$W^u(z_\varepsilon^-(t_0^*))$ and $W^s(z_\varepsilon^+(t_0^*))$ intersect
transversally at some point, $z_h\in\Sigma$, $\varepsilon$-close to
$z_0\in\Sigma$ and
\begin{equation*}
\left\{\phi(t;t_0^*,z_h,\varepsilon),\,t\in\RR\right\},
\end{equation*}
is a heteroclinic orbit between the periodic orbits
$\Lambda^-_\varepsilon$ and $\Lambda^+_\varepsilon$.
\item If $r<1$, there exists $\rho$ such that, given $\te,\tr>0$ satisfying
$0<\frac{\tr}{\te}<\rho$, one can find $\delta_0$ such that, if
$\varepsilon=\te\delta$ and $r=1-\tr\delta$, then, for $0<\delta<\delta_0$,
there exists a simple zero of the function $\Delta(t_0,\te\delta,1-\tr\delta)$
of the form $t_0^*=\bt+O(\frac{\tr}{\te})+O(\delta)$. Hence, the curves
$W^u(z_\varepsilon^-(t_0^*))$ and $W^s(z_\varepsilon^+(t_0^*))$ intersect
$\Sigma$ transversally at two points, $z_h^\pm\in\Sigma$, satisfying
$z_h^+=z_0+O(\delta)$ and $z_h^-=z_h^+/r$, such that
\begin{equation*}
\left\{ \phi(t;t_0^*,z_h^+,\te\delta,1-\tr\delta),\,t\in\RR \right\}
\end{equation*}
is a heteroclinic orbit between the periodic orbits
$\Lambda^-_\varepsilon$ and $\Lambda^+_\varepsilon$.
\end{enumerate}
\end{thm}
\begin{remark}
Note that, for $r=1$, we recover the classical result given by the Melnikov
method for heteroclinic orbits extended to the non-smooth
system~(\ref{eq:general_field}).
\end{remark}

\begin{proof}
Applying the fundamental theorem of calculus to the functions
\begin{equation*}
s\longmapsto H_0^{\pmm}\left( \phi^\pmm\left( s;t_0,z^{s/u},\varepsilon \right) \right),
\end{equation*}
we obtain
\begin{align*}
H_0^\pmm\left( z^{s/u} \right)=H_0^\pm(\phi\left( T^{s/u};t_0,z^{s/u},\varepsilon
\right)+\int_{T^{s/u}}^{t_0}\frac{d}{ds}H_0^\pmm\left( \phi^\pmm\left(
s;t_0,z^{s/u},\varepsilon
\right)ds \right),
\end{align*}
and then make $T^{s/u}=\pmm\infty$. However, the limits
\begin{equation*}
\lim_{t\to\pmm\infty}\phi^\pmm(t;t_0,z^{s/u},\varepsilon)
\end{equation*}
do not exist because the flow at the respective stable/unstable manifolds tends
to the periodic orbit $\Lambda^\pm_{\varepsilon}$. To avoid this limit, we
proceed as follows.\\
Given $t_0$, we define
\begin{equation}
\begin{aligned}
f_-(s)&=H_0^-\left( \phi^-\left( s;t_0,z^u,\varepsilon \right)
\right)-H_0^-\left( \phi^-\left( s;t_0,z^-_\varepsilon(t_0),\varepsilon \right)
\right),\,s\le t_0\\
f_+(s)&=H_0^+\left( \phi^+\left( s;t_0,z^s,\varepsilon \right)
\right)-H_0^-\left( \phi^+\left( s;t_0,z^+_{\varepsilon}(t_0),\varepsilon \right)
\right),\,s\ge t_0,
\end{aligned}
\label{eq:aux_functions}
\end{equation}
which are well defined smooth functions because the flow is restricted to the stable and
unstable invariant mani\-folds or to the hyperbolic periodic orbit and never crosses the switching
manifold~$\Sigma$.\\
Then, we write Eq.~(\ref{eq:distance_zus}) as
\begin{equation}
\Delta(t_0,\varepsilon,r)=r^2f_-(t_0)-f_+(t_0)+r^2H_0^-\left( z^-_\varepsilon\left( t_0
\right) \right)-H_0^+\left( z^+_\varepsilon\left( t_0 \right) \right).
\label{eq:distance_zus5}
\end{equation}
Noting that
\begin{equation}
H_0^\pm(z_\varepsilon^\pm(t_0))=\underbrace{H_0^\pm(z^\pm)}_{c_1}+\varepsilon\underbrace{DH_0^\pm(z^\pm)}_{\stackrel{\shortparallel}{0}}\frac{\partial
z_\varepsilon^\pm(t_0)}{\partial \varepsilon}{|_{\varepsilon=0}}+O(\varepsilon^2),
\label{eq:H_at_xeps}
\end{equation}
Eq.~(\ref{eq:distance_zus5}) becomes
\begin{equation}
\Delta(t_0,\varepsilon,r)=r^2f_-(t_0)-f_+(t_0)+(r^2-1)c_1+O(\varepsilon^2)
\label{eq:distance_zus6}
\end{equation}
We apply the fundamental theorem of calculus to the
functions~(\ref{eq:aux_functions}) to compute
\begin{equation}
\begin{aligned}
f_-(t_0)&=f_-(T^u)+\int_{T^u}^{t_0}f'_-(s)ds=\\
&f_-(T^u)+\varepsilon\int_{T^u}^{t_0}\Big(\left\{ H_0^-,H_1^-
\right\}\left( \phi^-\left( s;t_0,z^u,\varepsilon \right),s \right)\\
&-\left\{H_0^-,H_1^- \right\}\left( \phi^-\left( s;t_0,z^-_{\varepsilon}\left( t_0
\right),\varepsilon \right),s \right)\Big)ds\\
f_+(t_0)&=f_+(T^s)-\int^{T^s}_{t_0}f'_+(s)ds=\\
&f_+(T^s)-\varepsilon\int^{T^s}_{t_0}\Big(\left\{ H_0^+,H_1^+
\right\}\left( \phi^+\left( s;t_0,z^s,\varepsilon \right),s \right)\\
&-\left\{H_0^+,H_1^+ \right\}\left( \phi^+\left( s;t_0,z^+_{\varepsilon}\left( t_0
\right),\varepsilon \right),s \right)\Big)ds.
\end{aligned}
\label{eq:faux2}
\end{equation}
Due to the hyperbolicity of the periodic orbits $\Lambda^\pm_{\varepsilon}$, the
flow on $W^{s/u}(\Lambda_\varepsilon^\pmm)$ converges exponentially to them (forwards
or backwards in time). That is, there exist positive constants $C$, $\lambda$
and $s_0$ such that
\begin{equation*}
\Big\vert\phi^+\left( s;t_0,z^{s},\varepsilon
\right)-\phi^+\left( s;t_0,z^+_\varepsilon(t_0),\varepsilon
\right)\Big\vert<C e^{-\lambda s},\;\forall s>s_0,
\end{equation*}
and similarly for $\phi^-$. This allows one to make $T^{s/u}\to\pmm\infty$ in
Eqs.~(\ref{eq:faux2}), since
\begin{equation*}
\lim_{s\to\pm\infty} f_{\pm}(s)=0
\end{equation*}
and, moreover, the improper integrals converge in the limit.\\
Now, expanding the expressions in~(\ref{eq:faux2}) in powers of $\varepsilon$, we find
\begin{equation}
\begin{aligned}
&f_-(t_0)=\varepsilon\int_{-\infty}^{t_0}\left\{ H_0^-,H_1^-
\right\}\left( \phi^-\left( s;t_0,z_0,0 \right),s \right)ds+O(\varepsilon^2)\\
&f_+(t_0)=-\varepsilon\int^{\infty}_{t_0}\left\{ H_0^+,H_1^+
\right\}\left( \phi^+\left( s;t_0,z_0,0 \right),s \right)ds+O(\varepsilon^2),
\label{eq:fauxexp}
\end{aligned}
\end{equation}
where we have used property~(\ref{eq:H_at_xeps}) to include the second terms in
the integrals into the higher order terms. Finally, substituting Eq.~(\ref{eq:fauxexp}) into
Eq.~(\ref{eq:distance_zus6}), we obtain
\begin{equation}
\Delta(t_0,\varepsilon,r)=(r^2-1)c_1+\varepsilon
M(t_0)+O(\varepsilon^2)+O(\varepsilon\left( r-1 \right)),
\label{eq:distance_zus4}
\end{equation}
where $M(t_0)$ is defined in Eq.~(\ref{eq:Melnikovs_function}). 

We now distinguish between the cases $r=1$ and $r<1$. If $r=1$, we recover 
the classical expression for the distance between the perturbed invariant 
manifolds. By applying the implicit function theorem, it is easy to show 
that, if $M(t_0)$ has a simple zero at $\bt$, then
$\Delta(t_0,\varepsilon,1)$ has a simple zero at $t_0^*=\bt+O(\varepsilon)$. Thus,
the curves $W^u(z_\varepsilon^-(t_0^*))$ and $W^s(z_\varepsilon^+(t_0^*))$
intersect $\Sigma$ transversally at some point, $z_h=z^u(t_0^*)=z^s(t_0^*)\in\Sigma$, $\varepsilon$-close to
$z_0\in\Sigma$. Therefore,
\begin{equation*}
\left\{\phi(t;t_0^*,z_h,\varepsilon),\,t\in\RR\right\},
\end{equation*}
is a heteroclinic orbit between the periodic orbits
$\Lambda^-_\varepsilon$ and $\Lambda^+_\varepsilon$.\\

If $r<1$, we define $\varepsilon=\te\delta$ and $r=1-\tr\delta$, and
Eq.~(\ref{eq:distance_zus4}) becomes
\begin{equation}
\frac{\Delta(t_0,\te\delta,1-\tr\delta)}{\delta}=-2\tr c_1+\te M(t_0)+O(\delta).
\label{eq:distance_zus7}
\end{equation}
Then we argue as in Theorem~\ref{theo:dissipative_subharmonic_orbits}.  As
$M(t_0)$ is a smooth periodic function, it possesses at least one local maximum.
Let $t_M$ be the closest value to $\bt$ where $M(t_0)$ possesses a local
maximum, and assume $M'(t_0)\ne0$ for all $t_0$ between $\bt$ and $t_M$. If
$M'(t_0)$ vanishes between $\bt$ and $t_M$, we then take $t_M$ to be the closest
value to $\bt$ such that $M'(t_0)=0$ to ensure that $M'(t_0)\ne0$ between $\bt$
and $t_M$. We then define $\rho:=\frac{M(t_M)}{2c_1}$. Then, if
\begin{equation*}
0<\frac{\tr}{\te}<\rho,
\end{equation*}
there exists $\wt$ $\frac{\tr}{\te}$-close to $\bt$ such that
\begin{equation*}
-2\tr c_1+M(\wt)=0
\end{equation*}
and $M'(\wt)\ne0$. Hence, we can apply the implicit function theorem to
Eq.~(\ref{eq:distance_zus7}) near the point $(t_0,\delta)=(\wt,0)$ to conclude
that there exists $\delta_0$ such that, if $0<\delta<\delta_0$, then one can
find
\begin{equation*}
t_0^*=\wt+O(\delta)=\bt+O(\delta)+O(\tr/\te)
\end{equation*}
which is a simple solution of Eq.~(\ref{eq:distance_zus7}).\\
Hence, arguing similarly as for $r=1$, there exist two points
$z_h^+=z^s(t_0^*)=z_0+O(\delta)$ and $z_h^-=z^u(t_0^*)=z_0/r+O(\delta)r$ such
that $z_h^+=rz_h^-$ and
\begin{equation*}
\left\{ \phi(t;t_0^*,z_h^+,\te\delta,1-\tr\delta),\,t\in\RR \right\}
\end{equation*}
where
\begin{equation*}
\phi\left( t;t_0^*,z_h^+,\varepsilon,r \right)=\left\{
\begin{aligned}
&\phi^-\left( t;t_0,t_0^*,z_h^+/r,\varepsilon \right)&&\text{if }t\le
t_0^*\\
&\phi^+\left( t;t_0,t_0^*,z_h^+,\varepsilon \right)&&\text{if }t\ge
t_0^*
\end{aligned}\right.
\end{equation*}
is a heteroclinic orbit between the periodic orbits
$\Lambda^-_\varepsilon$ and $\Lambda^+_\varepsilon$.
\end{proof}

\section{Example: the rocking block}\label{sec:rocking_block}
\subsection{System equations}
In order to illustrate the results shown in the previous sections, we consider
the mechanical system shown in Fig.~\ref{fig:rocking_block}, which consists of a
rocking block under a horizontal periodic forcing given by
\begin{equation}
a_H(t)=\varepsilon\alpha g \cos\left( \Omega t+\theta \right).
\label{eq:periodic_forcing}
\end{equation}
\begin{figure}
\begin{center}
\includegraphics[width=0.6\textwidth]{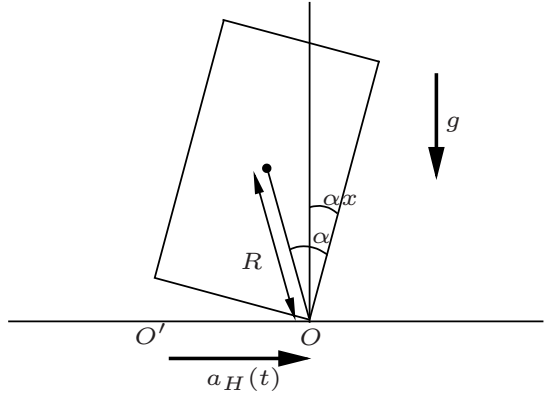}
\end{center}
\caption{Rocking block}
\label{fig:rocking_block}
\end{figure}
\noindent
This system was first studied in \cite{Hou63}. The equations
that govern its behaviour are well known (see for example
\cite{YimChoPen80,SpaKoh84}), and are given by
\begin{align}
\alpha\ddot{x}+\sgn(x)\sin( \alpha( 1&-\sgn\left( x \right)x))=\nonumber\\
&-\alpha\varepsilon\cos\left( \alpha\left( 1-\sgn\left( x \right)x \right) \right)\cos\left(
\omega t \right)\label{eq:nonlinear_diferential_equation}\\
&\dot{x}(t^+_A)=r\dot{x}(t^-_A) \hspace{5mm} (x=0)
\label{eq:loss_of_energy}
\end{align}
where the last equation,~(\ref{eq:loss_of_energy}), simulates the loss of
energy of the block at every impact with the ground, as described in
\S\ref{sec:rest_coeff}. In addition, the function
\begin{equation}
\sgn(x)=\left\{
\begin{aligned}
1&&\text{if }x>0\\
-1&&\text{if }x<0
\end{aligned}\right.
\label{eq:sign}
\end{equation}
distinguishes between the two modes of movement: rocking about the point $O$
when the angle is positive ($x>0$) or rocking about $O'$ when the angle $x$ is
negative. Obviously, this makes the system piecewise smooth and so it can be
written in the form of Eq.~(\ref{eq:general_field}). Moreover, conditions
\CondsC~ are satisfied and, hence, the results shown in previous sections can be applied.
However, as our purpose here is to illustrate them, we will consider the linearized version of
Eq.~(\ref{eq:nonlinear_diferential_equation}) instead, which will permit us to
perform explicit analytical computations. This linearization is achieved by
assuming $\alpha\ll1$, namely that the block is slender \cite{Hog89}. Thus, the
system that we will consider, written in the form of
Eq.~(\ref{eq:general_field}), is
\begin{align}
\left.
\begin{aligned}
\dot{x}=&y\\
\dot{y}=&x-1-\varepsilon\cos\left( \omega t \right)
\end{aligned}
\right\}&\text{ if }x>0\label{eq:linear_positive}\\
\nonumber\\
\left.
\begin{aligned}
\dot{x}=&y\\
\dot{y}=&x+1-\varepsilon\cos\left( \omega t  \right)
\end{aligned}
\right\}&\text{ if }x<0\label{eq:linear_negative}\\
\nonumber\\
y(t^+_A)=ry(t^-_A)&\hspace{5mm} (x=0)
\label{eq:loss_of_energy2},
\end{align}
where the perturbation becomes a smooth function due to the linearization.\\
If $r=1$, system~(\ref{eq:linear_positive})-(\ref{eq:linear_negative}) can be written in
the form~(\ref{eq:hamiltonian_properties}) using the Hamiltonian function
\begin{equation}
H_\varepsilon(x,y,t)=H_0(x,y)+\varepsilon H_1(x,t),
\label{eq:hamiltonian_linear}
\end{equation}
where
\begin{equation}
H_0(x,y)=\left\{
\begin{aligned}
\frac{y^2}{2}-\frac{x^2}{2}+x,&\text{ if }x>0\\
\frac{y^2}{2}-\frac{x^2}{2}-x,&\text{ if }x<0
\end{aligned}\right.
\label{eq:linear_unperturbed-Hamiltonian}
\end{equation}
and
\begin{equation}
H_1(x,t)=x\cos\left( \omega t \right)
\label{eq:linear_Hamitlonian_perturbation}
\end{equation}
is $T$-periodic, with $T=2\pi/\omega$ and is a $C^\infty$ function.\\
In addition, when $\varepsilon=0$, conditions \CondsC~are fulfilled, and the
phase portrait for the
system~(\ref{eq:linear_positive})-(\ref{eq:linear_negative}) is equivalent to
the one shown in Fig.~\ref{fig:phase_portrait_unperturbed}. That is, it
possesses an invisible fold-fold of centre type at the origin and two saddle
points at $(1,0)$ and $(-1,0)$ connected by two heteroclinic orbits.
Furthermore, the origin is surrounded by a continuum of periodic orbits whose
periods monotonically increase as they approach to the heteroclinic connections.
Using Eqs.~(\ref{eq:half_period_unperturbed_orbit}) and
(\ref{eq:period_unperturbed_orbit}), the symmetries of the
Hamiltonian~(\ref{eq:linear_unperturbed-Hamiltonian}) and assuming $y_0>0$,
these periods are given by
\begin{align}
\alpha(y_0)&=4\int_0^{1-\sqrt{1-y_0^2}}\frac{1}{\sqrt{y_0^2+x^2-2x}}dx=\nonumber\\
&=2\ln\left(\frac{1+y_0}{1-y_0}\right),\label{eq:period_unperturbed_linear_orbit}
\end{align}
and hence $\alpha'(y_0)>0$.

\subsection{Existence of periodic orbits}
We first study the persistence of $(n,m)$-periodic orbits for $r=1$ in
Eq.~(\ref{eq:loss_of_energy2}) by applying
Theorem~\ref{theo:subharmonic_melnikov_conservative}. The subharmonic Melnikov
function~(\ref{eq:subharmonic_Melnikov_function}) becomes
\begin{equation}
M^{n,m}(t_0)=-\int_{0}^{nT}\Pi_y(q_c(t))\cos(\omega\left(t+t_0\right))dt,
\label{eq:M1}
\end{equation}
where $q_c(t)$ is the periodic orbit of the unperturbed version of
system~(\ref{eq:linear_positive})-(\ref{eq:linear_negative}) with Hamiltonian
$c=\frac{\by^2}{2}$ satisfying $q_{c}(0)=(0,\by)$ and 
\begin{equation}
\by=\alpha^{-1}\left(\frac{nT}{m}\right)=\frac{e^{\frac{nT}{2m}}-1}{e^{\frac{nT}{2m}}+1}.
\label{eq:by_linear_example}
\end{equation}
We now want to obtain an explicit expression for Eq.~(\ref{eq:M1}). Thus we
first note that the solution of
system~(\ref{eq:linear_positive})-(\ref{eq:linear_negative}) with initial
condition $(x_0,y_0)$ at $t=t_0$ is given by
\begin{align}
x^\pm(t)&=C_1^\pm e^{t}+C_2^\pm e^{-t}\pm1
\\
y^\pm(t)&=C_1^\pm e^{t}-C_2^\pm e^{-t},
\label{eq:linear_solution}
\end{align}
where
\begin{equation}
C_1^\pm=\frac{x_0+y_0
\mp1}{2}e^{-t_0},\;C_2^\pm=\frac{x_0-y_0\mp1}{2}e^{t_0}.
\label{eq:solutions:constants}
\end{equation}
As explained in~\S\ref{sec:definitions}, the superscript
$+$ is applied if $x_0>0$ or $x_0=0$ and $y_0>0$, and the $-$ otherwise.\\
Assuming $x_0=0$ and $y_0=\by>0$, an expression for $\Pi_y(q_c(t))$ becomes
\begin{equation}
\Pi_y(q_c(t))=\left\{
\begin{aligned}
C_1e^t-C_2e^{-t},&&&\text{ if }0\le t
\le \frac{nT}{2m}\\
-C_1e^{t-\frac{nT}{2m}}+C_2e^{-t+\frac{nT}{2m}},&&&\text{ if }\frac{nT}{2m}\le t
\le \frac{nT}{m},
\end{aligned}\right.
\label{eq:periodi_orbit_linear}
\end{equation}
where
\begin{equation}
C_1=\frac{\by-1}{2},\;C_2=\frac{-\by-1}{2}.
\label{eq:solutions_constants2}
\end{equation}
Thus, Eq.~(\ref{eq:M1}) becomes
\begin{align*}
&M^{n,m}(t_0)=-\sum_{j=0}^{m-1}\Bigg(\int_{0}^{\frac{nT}{2m}}\left(C_1e^t-C_2e^{-t}\right)\cos\left(
\omega\left( t+t_0+j\frac{nT}{m} \right) \right)dt\\
&+\int_{\frac{nT}{2m}}^{\frac{nT}{m}}\left(-C_1e^{t-\frac{nT}{2m}}+C_2e^{-t+\frac{nT}{2m}}\right)\cos\left(
\omega\left( t+t_0+j\frac{nT}{m} \right) \right)dt\Bigg)
\end{align*}
and, after some computations, we have
\begin{equation}
M^{n,m}(t_0)=\left\{
\begin{aligned}
-\frac{4}{\omega^2+1}\cos\left( \omega t_0 \right),&&&\text{ if }m=1\\
0,&&&\text{ if }m>1.
\end{aligned}
\right.
\label{eq:M2}
\end{equation}
As $M^{n,1}(t_0)$ has two simple zeros, $\bt^1=\frac{T}{4}$ and
$\bt^2=\frac{3T}{4}$, by Theorem~\ref{theo:subharmonic_melnikov_conservative}, 
if $\varepsilon>0$ is small enough, the non-autonomous
system~(\ref{eq:linear_positive})-(\ref{eq:linear_negative}) possesses two
subharmonic $(n,1)$-periodic orbits. In addition, the
initial conditions of these periodic orbits are
$\varepsilon$-close to
\begin{equation}
(0,\by,\bt^1)=(0,\frac{e^{\frac{nT}{2}}-1}{e^{\frac{nT}{2}}+1},\frac{T}{4})
\label{eq:seed1}
\end{equation}
and
\begin{equation}
(0,\by,\bt^2)=(0,\frac{e^{\frac{nT}{2}}-1}{e^{\frac{nT}{2}}+1},\frac{3T}{4}),
\label{eq:seed2}
\end{equation}
respectively.\\
Proceeding as in Remark~\ref{rem:algorithm}, one can solve numerically Eq.~(\ref{eq:periodic_orbit_condition}) with $m=1$
and find the initial conditions for such a periodic orbit.  In
Fig.~\ref{fig:5-1_po} we show the result of that for $n=5$. Each periodic orbit
is obtained by using the points given in Eqs.~(\ref{eq:seed1}) and
(\ref{eq:seed2}) to initiate the Newton method. Then, tracking the obtained
solution, $\varepsilon$ has been increased up to
$\varepsilon=1.6565\cdot10^{-2}$.\\
\begin{figure}
\begin{center}
\includegraphics[angle=-90,width=0.7\textwidth]{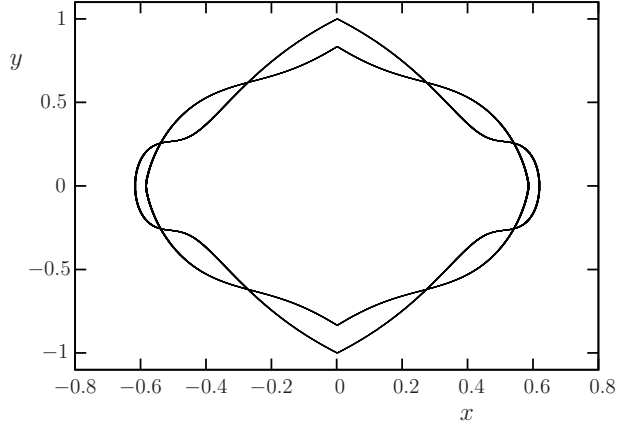}
\end{center}
\caption{Periodic orbits for $n=5$ and $m=1$, $\omega=5$ and
$\varepsilon=1.6565\cdot10^{-2}$. Their initial conditions are
$\varepsilon$-close to the points given in Eqs.~(\ref{eq:seed1}) and
(\ref{eq:seed2}).}
\label{fig:5-1_po}
\end{figure}
Regarding the existence of $(n,m)$-periodic orbits with $m>1$ (ultrasubharmonic
orbits), as the subharmonic Melnikov function is identically zero nothing can be
said using the first order analysis that we have shown in this work.\\
However, if instead~(\ref{eq:linear_Hamitlonian_perturbation}) one considers the
perturbation
\begin{equation*}
H_1(x,t)=x\left(\cos\left( \omega t \right)+\cos\left( k\omega t \right)\right),
\end{equation*}
then, it can be seen that the corresponding Melnikov function possesses simple
zeros for $m=k$ and $n$ relatively prime odd integers. Thus, periodic orbits
impacting $m>1$ times with the switching manifold can exist if higher harmonics of the
perturbation are considered.\\

Let us now introduce the energy dissipation described in
\S\ref{sec:dissipative_subharmonic} and consider the whole
system~(\ref{eq:linear_positive})-(\ref{eq:loss_of_energy2}) with $r<1$ using
the Hamiltonian perturbation~(\ref{eq:hamiltonian_linear}). From
Theorem~\ref{theo:dissipative_subharmonic_orbits}, simple zeros of the
Melnikov function~(\ref{eq:M2}) also guarantee the existence of $(n,1)$-periodic
orbits when $1-r$ is small enough with respect to $\varepsilon$. More precisely,
taking
\begin{equation}
\varepsilon=\te\delta,\;r=1-\tr\delta,
\label{eq:change_of_parameters2}
\end{equation}
condition~(\ref{eq:dissipative_nm_periodic_orbit_existence_condition1}) becomes
\begin{equation}
0<\frac{\tr}{\te}<\frac{1}{2}\left(\frac{e^{\frac{nT}{2}}+1}{e^{\frac{nT}{2}}-1}\right)^2M^{n,1}(t_M):=\rho,
\label{eq:dissipative_po_condition}
\end{equation}
where $M^{n,1}(t_M)=M^{n,1}(\frac{T}{2})=\frac{4}{\omega^2+1}$ is the maximum
value of the Melnikov function~(\ref{eq:M2}). Then there exists an
$(n,1)$-periodic orbit if $\delta>0$ is small enough. The initial condition
of the periodic orbit is located in a $\delta$-neighbourhood of the point
$(x_0,y_0,t_0)=(0,\by,\wt)$, where $\by$ is defined in
Eq.~(\ref{eq:by_linear_example}), such that
\begin{equation*}
\alpha(\by)=nT
\end{equation*}
and $\wt$ is given by the simple zeros of Eq.~(\ref{eq:wt_equation}), which becomes
\begin{equation}
-2\tr\by^2+\te M^{n,1}(t_0)=0.
\label{eq:wt_equation_linear_example}
\end{equation}
Hence we find
\begin{equation}
\wt^i=\frac{1}{\omega}\arccos\left( -\frac{\omega^2+1}{2}
\left(\frac{e^{\frac{nT}{2}}-1}{e^{\frac{nT}{2}}+1}\right)^2\frac{\tr}{\te}\right)+(i-1)\frac{T}{2},\;i=1,2.
\label{eq:wt_solutions_linear}
\end{equation}

\unitlength=\textwidth

\begin{figure}\centering
\subfigure[$\delta=1.04$]
{\includegraphics[angle=-90,width=0.45\textwidth]{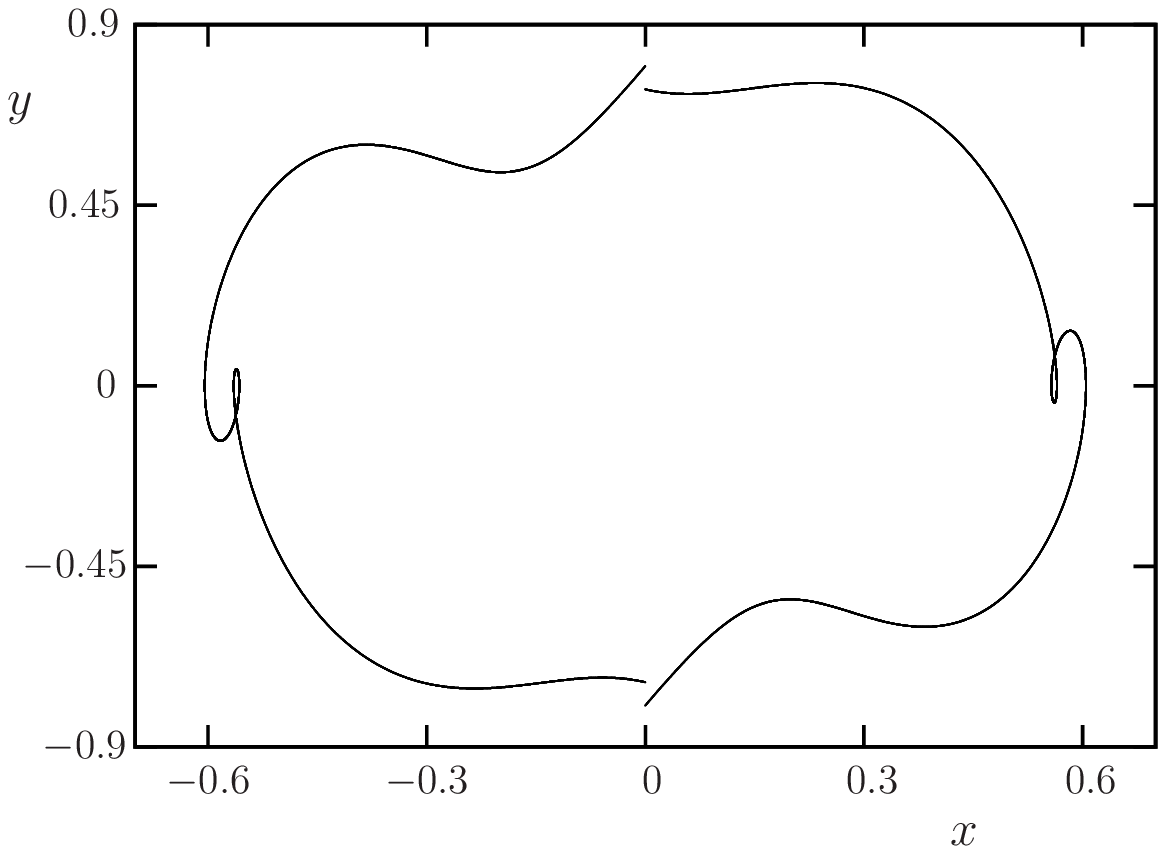}}
\subfigure[$\delta=4.125$]
{\includegraphics[angle=-90,width=0.45\textwidth]{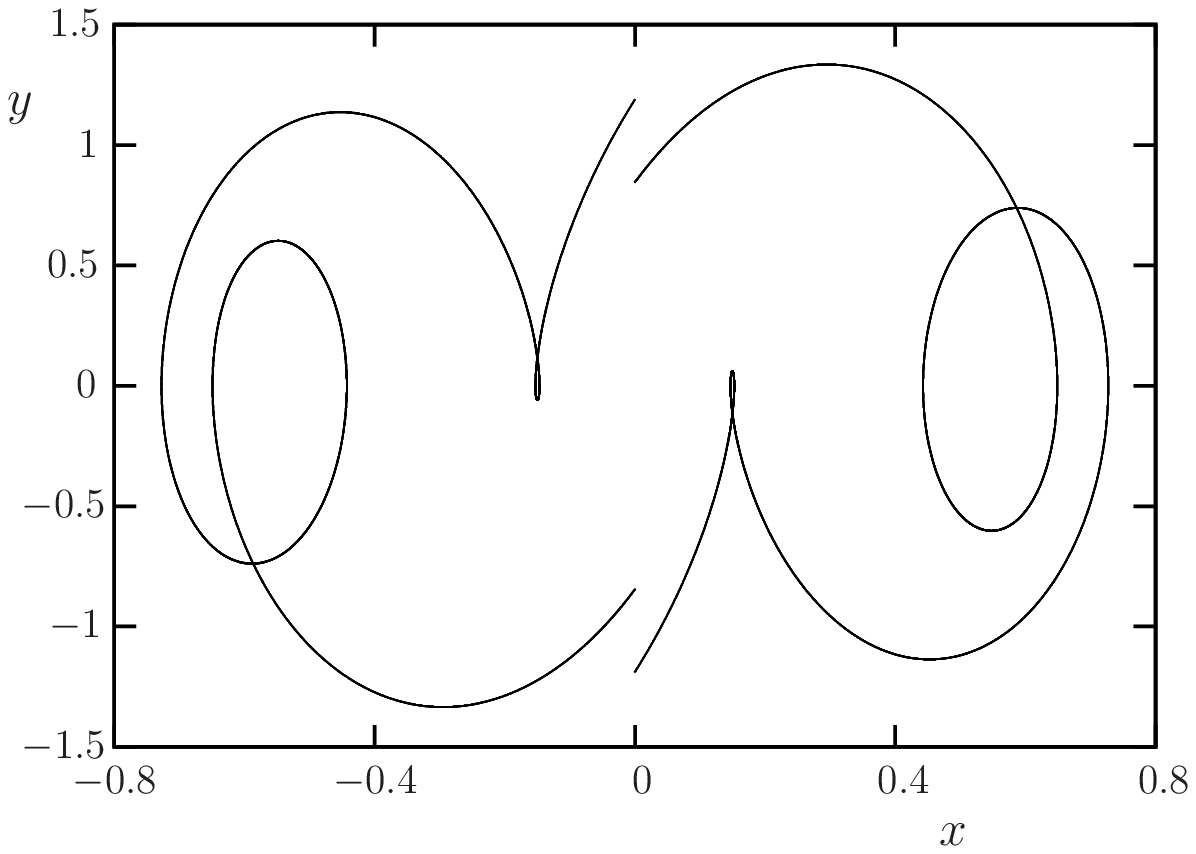}}
\caption{$(5,1)$-periodic orbits for $\omega=5$ and $\frac{\tr}{\te}=0.07$.
Tracking the obtained solution, the perturbation parameter $\delta$ has been
increased up to its maximum value.  Initial conditions close to $(\by,\wt^1)$
and $(\by,\wt^2)$ have been used in (a) and (b),
respectively.}\label{fig:dissipative_po_omega5}
\end{figure}
As before, we set  $n=5$ and $\omega=5$. Then
expression~(\ref{eq:dissipative_po_condition}) becomes
\begin{equation}
0<\frac{\tr}{\te}<0.0914.
\label{eq:5-1_dissipative_po_condition}
\end{equation}
Hence, for any fixed ratio $\frac{\tr}{\te}$
satisfying~(\ref{eq:5-1_dissipative_po_condition}) there exist two points,
$(\by,\wt^i)$, $i=1,2$, such that, if $\delta$ is small enough,
Eq.~(\ref{eq:dissipative_nm_periodic_orbit_condition3}) possesses a solution
$\delta$-close to them. Such a solution is an initial condition for an
$(n,1)$-periodic orbit of
system~(\ref{eq:linear_positive})-(\ref{eq:loss_of_energy2}), with
$r=1-\tr\delta$ and $\varepsilon=\te\delta$.

In Fig.~\ref{fig:dissipative_po_omega5} some of these orbits are shown for one
value of the ratio satisfying~(\ref{eq:5-1_dissipative_po_condition}).  Two different
periodic orbits are shown, which are the ones whose initial conditions are
$\delta$-close to the values $(\by,\wt^1)$ and $(\by,\wt^2)$.  In both cases,
$\delta$ tracks the solution, up to the value from
which solutions of Eq.~(\ref{eq:dissipative_po_condition}) can no longer be
found.  These are the values used in the simulations shown in
Fig.~\ref{fig:dissipative_po_omega5}. Above
the limiting value of the ratio given in~(\ref{eq:5-1_dissipative_po_condition}), 
no $(5,1)$-periodic orbits were found for $\omega=5$.

\subsection{Existence curves}
We now use Theorem~\ref{theo:dissipative_subharmonic_orbits} to derive existence
curves for the $(n,1)$-periodic orbits ($n$ odd) and compare them with the ones obtained
in \cite{Hog89}. Unlike in \cite{Hog89}, we obtain these curves in the
$r$-$\varepsilon$ plane through the variation of $\delta$.\\
The limiting condition provided by
Theorem~\ref{theo:dissipative_subharmonic_orbits} is given in
Eq.~(\ref{eq:dissipative_po_condition}). Thus, for a given $r$ close to $1$
($\tr\delta$ close to 0), it is natural to arbitrarily fix $\tr$ and minimize
$\varepsilon$ by maximizing the ratio in~(\ref{eq:dissipative_po_condition}),
setting $\frac{\tr}{\te}=\rho$.  However, the upper boundary of $\delta$,
$\delta_0$, provided by Theorem~\ref{theo:dissipative_subharmonic_orbits} tends
to zero as $\frac{\tr}{\te}\to \rho$, as it is derived from the implicit
function theorem.  Thus, it is not possible to uniformally bound $\delta$ for
all the ratios between $0$ and $\rho$. Hence, the condition
$\frac{\tr}{\te}=\rho$ can not be used to derive a limiting relation between $r$
and $\varepsilon$. Instead, we proceed as follows.\\
We first fix $n$ odd and $\omega>0$. Then, for every ratio
$0<\frac{\tr}{\te}<\rho$, we increase $\delta$ from $0$ to $\delta_0$ by
numerically tracking the obtained solution using as initial seed one of the
values provided in Eq.~(\ref{eq:seed1}) or (\ref{eq:seed2}). This results in a
curve in the $r$-$\varepsilon$ plane parametrized by the ratio
$\frac{\tr}{\te}$.\\
Regarding the results obtained in \cite{Hog89}, such a curve was obtained analytically
and for global conditions, and has the expression
\begin{equation}
\epsmin(R)=\frac{\left( 1+\omega^2 \right)R\left( 1-\cosh\left(
\frac{nT}{2}
\right) \right)}{\sqrt{\omega^2\sinh^2\left( \frac{nT}{2} \right)R^2+\left( 2-R
\right)^2\left( 1+\cosh\left( \frac{nT}{2} \right) \right)^2}},
\label{eq:betamin}
\end{equation}
where $R=1-r$.\\
As our result is only locally valid, in order to compare both results we have to
check whether both curves are tangent at $\varepsilon=0$. From~(\ref{eq:betamin}) we easily obtain
\begin{equation*}
\epsmin'(0)=-\frac{1+\omega^2}{2}\left(
\frac{e^{\frac{nT}{2}}-1}{e^{\frac{nT}{2}}+1} \right)^2=-\frac{1}{\rho},
\end{equation*}
which, by the inverse function theorem, tells us that both curves are tangent at
$\varepsilon=0$.\\

In Fig.~\ref{fig:existence_curves}, we show an example for $n=5$ and $\omega=5$ 
using initial conditions near~(\ref{eq:seed1}).  As can be seen, the curve provides, 
for every value of $r$, both the maximum and minimum
values of $\varepsilon$ for which a $(5,1)$-periodic orbit exists. The lower boundary derived in \cite{Hog89},
$\left(\epsmin(\cdot)\right)\!^{-1}\!(\varepsilon)$ is also shown. As demonstrated
above, both curves are tangent at $\varepsilon=0$, with slope equal to $\rho$. Note that the
lower boundary does not coincide with the line $1-r=\rho\varepsilon$, although
their difference tends to zero as $r\to 1$. This confirms that one can not
derive the minimum value of $\varepsilon$ from condition
(\ref{eq:dissipative_po_condition}) for every fixed $r$. 

\begin{figure}
\begin{center}
\includegraphics[width=0.6\textwidth,angle=-90]{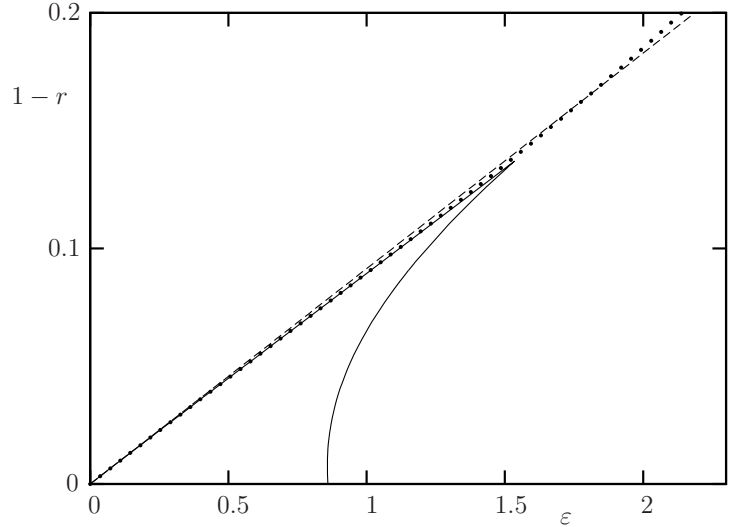}
\end{center}
\caption{Existence curves of a $(5,1)$-periodic orbit for $\omega=5$. Expression
derived from Theorem~\ref{theo:dissipative_subharmonic_orbits} (black line),
expression for $\varepsilon_{\text{min}}$ derived from \cite{Hog89} (dotted
curve) and line $1-r=\rho\varepsilon$ (dashed).}
\label{fig:existence_curves}
\end{figure}

\bibliographystyle{alpha}
\def\zh{Zh}\def\yu{Yu}\def\ya{Ya}

\end{document}